\documentclass[letterpaper,11pt]{article}
\usepackage{amsmath, amsthm, amssymb}    
\usepackage{bridges}			
\usepackage{graphicx}			
\usepackage[colorlinks=true, urlcolor=blue, citecolor=black, linkcolor=black]{hyperref}  
\usepackage{subcaption}			
\usepackage{amsfonts}

\newtheorem{theorem}{Theorem}

\DeclareMathOperator{\lcm}{lcm}

\title{Two-Disk Compound Symmetry Groups}

\author{Robert A. Hearn\textsuperscript{1}, William Kretschmer\textsuperscript{2}, \\
Tomas Rokicki\textsuperscript{3}, Benjamin Streeter\textsuperscript{4}, and Eric Vergo\textsuperscript{5} 
\vspace{10pt}\\
\textsuperscript{1}Gathering4Gardner; bob@hearn.to \textsuperscript{2}UT Austin; kretsch@cs.utexas.edu\\
\textsuperscript{3}Radical Eye Software; rokicki@gmail.com \textsuperscript{4}benpuzzles@gmail.com \textsuperscript{5}ericvergo@gmail.com}

\date{}					

\begin{document}

\maketitle

\thispagestyle{empty}

\begin{abstract}

Symmetry is at the heart of much of mathematics, physics, and art. Traditional geometric symmetry groups are defined in terms of isometries of the ambient space of a shape or pattern. If we slightly generalize this notion to allow the isometries to operate on overlapping but non-identical metric spaces, we obtain what we call \emph{compound symmetry groups}. A natural example is that of the groups generated by discrete rotations of overlapping disks in the plane. Investigation of these groups reveals a new family of fractals, as well as a rich structure that is intriguing both mathematically and artistically. We report on our initial investigations.
\end{abstract}

\section*{Introduction}

Symmetry is of fundamental importance in many disciplines of mathematics, from the fields of Galois theory, to the automorphisms of abstract algebra, to the isometries of wallpaper patterns and crystals. In physics, Noether's theorem connects the symmetries of space and time with conservation laws. In art, subtleties of degrees and kinds of symmetry arguably lie at the heart of what constitutes beauty.

Here we expand on the traditional notion of geometric symmetry, with mathematical and artistic consequences, at least. The \emph{symmetry group} of a shape or pattern is defined as the group of all isometries of the ambient space that preserve that shape or pattern. Isometries may only be combined in so many ways in any given metric space: there are (up to isomorphism) only 17 wallpaper groups, 7 frieze groups, 230 3-dimensional space groups, etc. We cannot have, for example, five-fold rotational symmetry in any repeating pattern in the plane---though quasicrystals can approximate this \cite{Senechal1996}.

\begin{figure}[h!tbp]
\centering
\begin{minipage}[c]{0.49\textwidth} 
	\includegraphics[width=\textwidth]{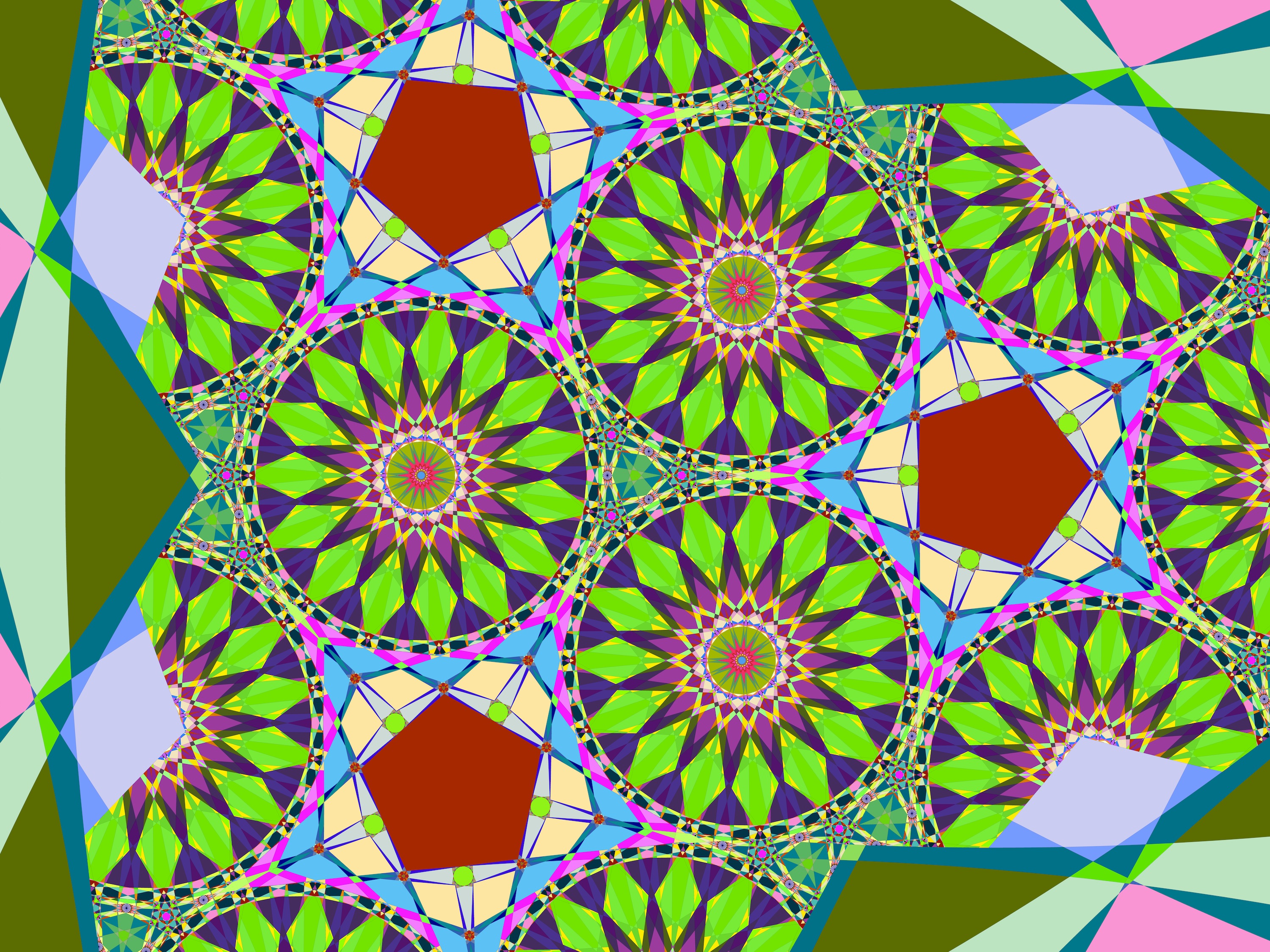}
	\subcaption{}
	\label{fig:3-5}
\end{minipage}
\hfill
\begin{minipage}[c]{0.49\textwidth} 
	\includegraphics[width=\textwidth]{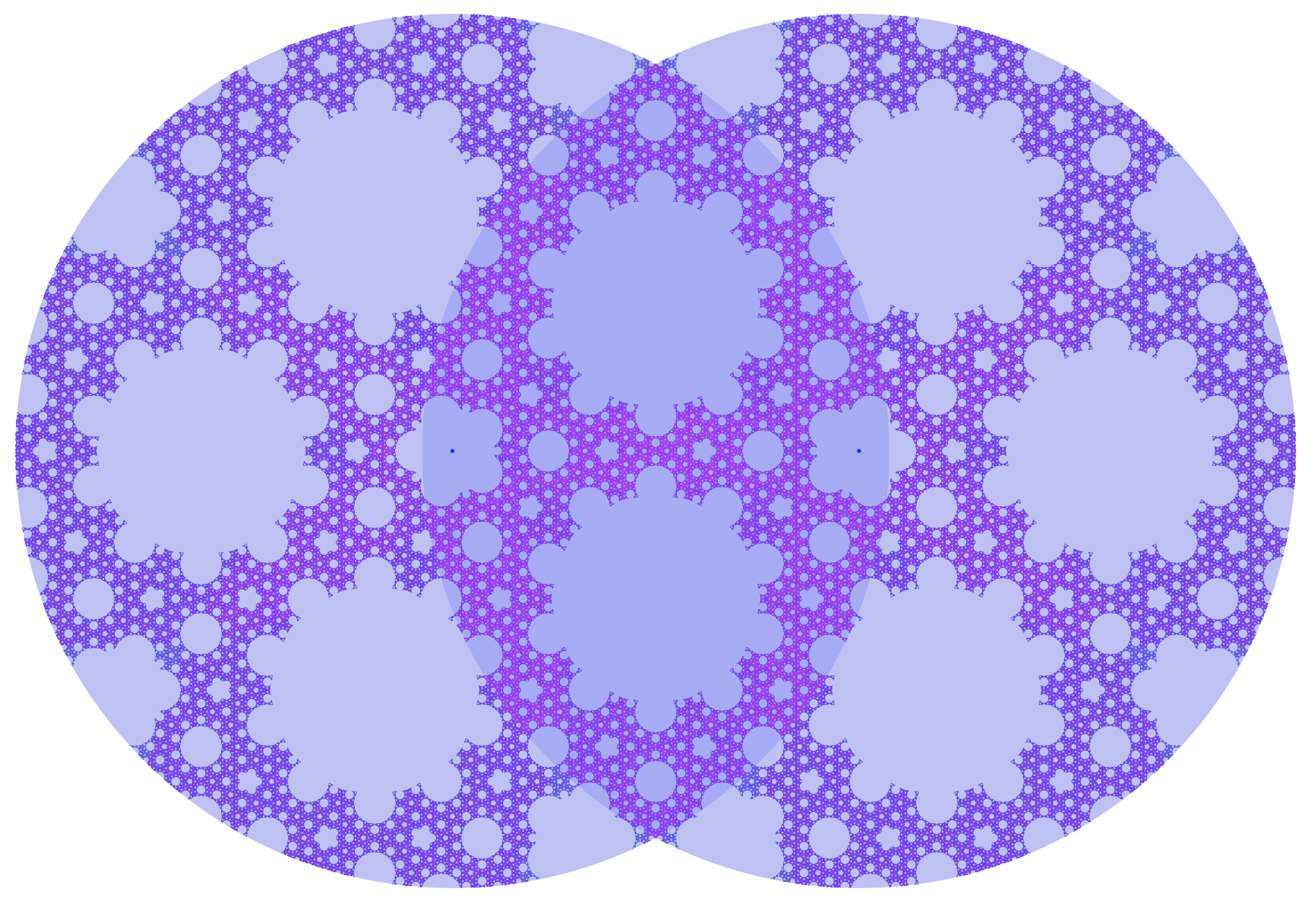}
	\subcaption{}
	\label{fig:n=5-fractal}
\end{minipage}
\caption{Images of compound symmetry groups:  (a) $GG_{3,5}(2.42,2.41)$, (b) the $n=5$ fractal.}
\end{figure}

By considering groups generated by isometries of metric subspaces that are not identical, but instead overlap, we can probe some of these ``forbidden'' symmetries in new ways. A portion of one of these ``compound symmetry groups'' is shown in Figure~\ref{fig:3-5}. Locally, this image contains regions of three-fold and five-fold symmetry---necessarily broken on larger scales---as well as combinations such as 15-fold.  A new family of fractals lies embedded within these groups at critical parameter values, as shown in Figure~\ref{fig:n=5-fractal}.

In this paper we explore the characteristics of these new kinds of symmetry group. Of particular importance will be determining the parameter values at which the groups become infinite, exploring the underlying dynamics, and also understanding the characteristic fractals, or sometimes pseudofractals, that appear precisely at these transitions.

\begin{figure}[h!tbp]
	\centering
	\includegraphics[width=.4\textwidth]{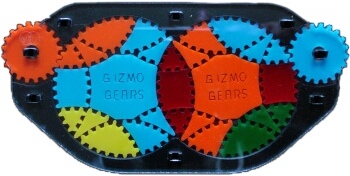}
    \caption{The Gizmo Gears puzzle.}
	\label{fig:gizmo}
\end{figure}

This work began with the study of the behavior of certain ``circle puzzles''~\cite{Engel1986, Engel2019}---especially \emph{Gizmo Gears}, shown in Figure~\ref{fig:gizmo}, designed by Doug Engel.\footnote{The specific question was this: can Gizmo Gears be finitely ``unbandaged'', or chopped into pieces so that all natural turns ($30^{\circ}$ in this case) are unblocked, regardless of configuration? If not, a puzzle is said to ``jumble''. Against all intuition at the time, Gizmo Gears does in fact jumble---it is past the critical radius (barely) for the compound symmetry group family $GG_{12}$. The original discussion can be found here: \url{https://twistypuzzles.com/forum/viewtopic.php?t=25752}.} A two-disk compound symmetry group is the mathematical generalization of a circle puzzle. The study of the group structure of circle puzzles seems to have been initiated in \cite{Kretschmer2017}.

\section*{Definitions and Basic Properties of Two-Disk Systems}

A \emph{compound symmetry group} is a group generated by a set of isometries of subspaces of a metric space.

Here we will primarily be concerned with compound symmetry groups generated by discrete rotations of two overlapping closed disks in the Euclidean plane. We sometimes call these \emph{two-disk systems}. Without loss of generality, let the two disks be centered at $(-1,0)$ and $(1, 0)$. Denote the left disk's radius as $r_1$, and the right disk's as $r_2$. The generators ${a, b}$ are rotation of the left disk by $-2\pi / n_1$, and of the right disk by $-2\pi / n_2$. The group operation is function composition on the left\footnote{We choose this convention so that move sequences can be read left to right. Similarly, the generators are defined to be clockwise rotations for compatibility with normal twisty puzzle notation.}: $ab(x)=b(a(x))$. We denote the group with these properties\footnote{$GG$ was chosen in honor of Gizmo Gears, and also, conveniently, to indicate the interaction of two groups.} as $GG_{n_1,n_2}(r_1,r_2)$. If $n_1=n_2$ we use a single subscript, and similarly for $r_1$ and $r_2$. We can also omit the radius specification to indicate a family of groups with unspecified but equal radii. For example, one very important family is $GG_5$---the groups generated by five-fold rotation of two equal disks.

To build some intuition about how two-disk systems work, consider Figure~\ref{fig:n=5-cases}. This figure shows the action of the group elements on points in the plane, for $GG_5$ at various $r$. Regions that remain connected under all elements (\emph{pieces}) are colored identically; the color is a function of the size of the orbit. In Figure~\ref{fig:n=5-1}, $r<1$, and the two rotations do not interact---the group is isomorphic to $C_5 \times C_5$. In Figure~\ref{fig:n=5-2}, the disks overlap, so there is some interaction. Viewed as a circle puzzle, we have added 9 pieces to the puzzle. This group is isomorphic to $C_5 \times C_5 \times A_9$ (we have added the even permutations of the wedge pieces). In Figure~\ref{fig:n=5-3}, the overlap has increased, and many more small pieces are created. Observe that in all cases, we do have five-fold rotational symmetry about two different points---but only within a fixed radius. 

\begin{figure}[h!tbp]
\centering
\begin{minipage}[c]{0.33\textwidth} 
	\includegraphics[width=\textwidth]{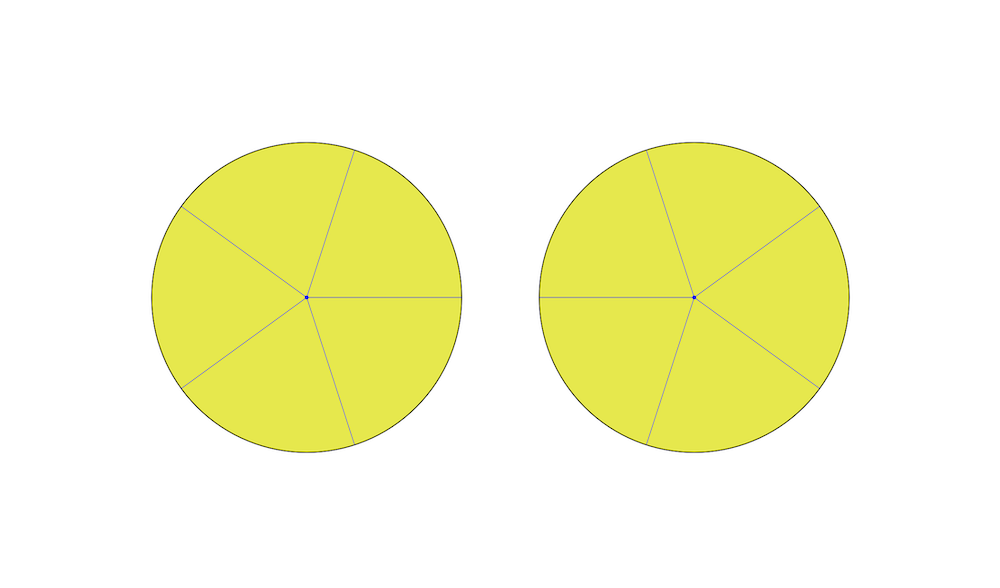}
	\subcaption{}
	\label{fig:n=5-1}
\end{minipage}
\hfill
\begin{minipage}[c]{0.33\textwidth} 
	\includegraphics[width=\textwidth]{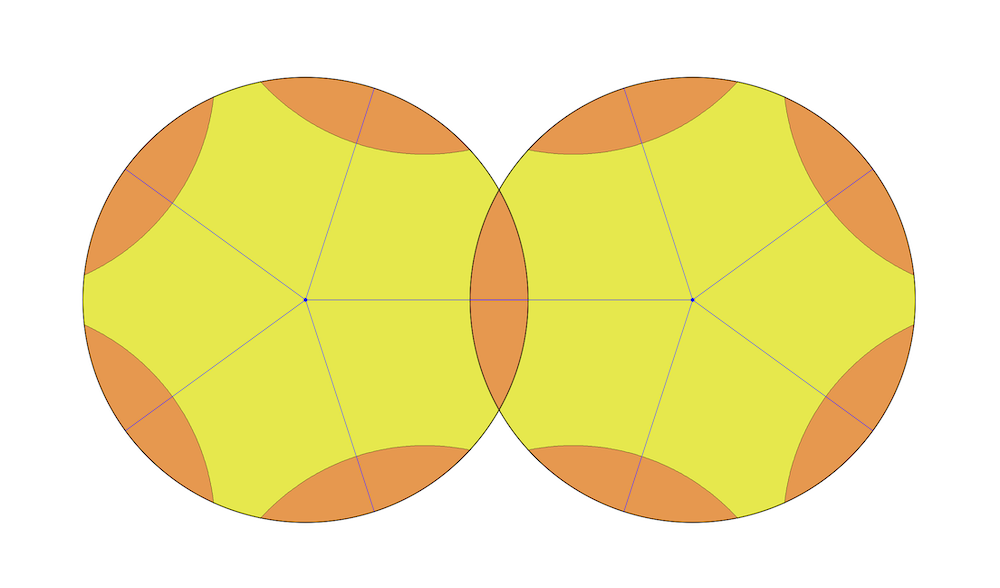}
	\subcaption{}
	\label{fig:n=5-2}
\end{minipage}
\hfill
\begin{minipage}[c]{0.33\textwidth} 
	\includegraphics[width=\textwidth]{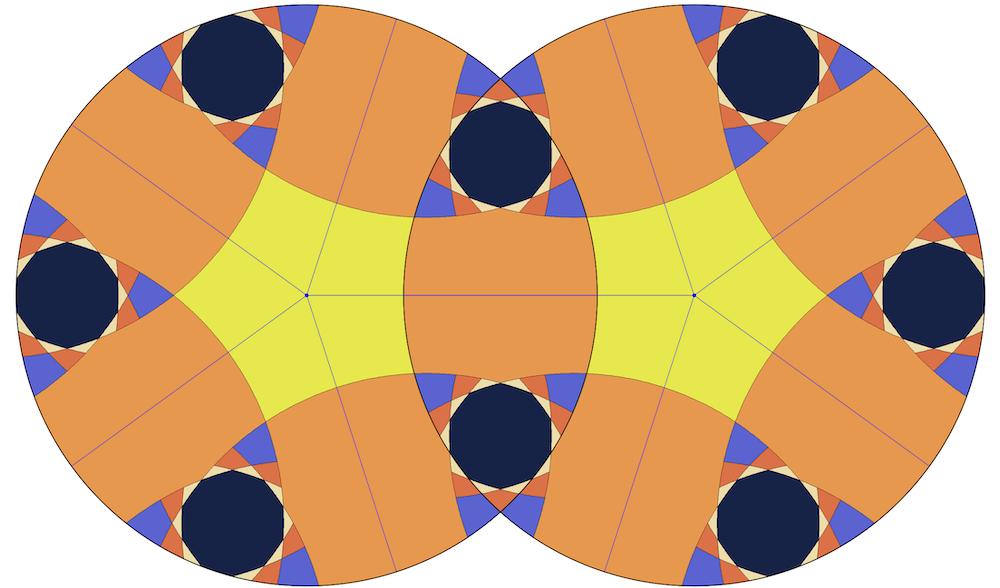}
	\subcaption{}
	\label{fig:n=5-3}
\end{minipage}
\caption{Different cases of $GG_5$. (a) is isomorphic to $C_5 \times C_5$; (b) is isomorphic to $C_5 \times C_5 \times A_9$; \\
(c) is more complicated.}
\label{fig:n=5-cases}
\end{figure}

It is important to note that while regular symmetry groups are generated by and consist of isometries, compound symmetry groups are likewise generated by isometries, but the general group element is \emph{not} an isometry: If we perform $ab$, different regions have been rotated by different amounts about different centers. This is called a \emph{piecewise isometry} \cite{Goetz2003}.

\subsection*{Infinite Groups}

A key question about any two-disk group is whether it is finite or infinite. If a family $GG$ has some infinite member, it will have a \emph{critical radius}, $r_c(GG)$, such that $GG(r)$ is finite when $r<r_c(GG)$, and infinite when $r>r_c(GG)$.\footnote{We believe, but have not proved, that it will also be infinite exactly at the critical radius.} We know this because the size of the orbit of any given point cannot decrease as $r$ increases---all group elements that affect it are still available---so $GG$ can never go from infinite to finite as $r$ increases.
We can also speak of the critical radius when $r_1\neq r_2$ if we fix one radius. 

We can precisely characterize which $GG_{n_1,n_2}$ have infinite members:

\begin{theorem}
\label{thm:not_2346_group_infinite}
There exists some $r$ for which $GG_{n_1,n_2}(r)$ is infinite if and only if $\lcm(n_1, n_2) \not\in \{2, 3, 4, 6\}$.\footnote{This fact is closely related to the crystallographic restriction theorem.}
\end{theorem}

\begin{proof}

We omit the ``only if'' proof in this paper, and prove the more interesting direction.

First, assume that $n_1=n_2=n$.
Observe that $a^{-1}b$ is a translation of all points moved by both rotations: The point at for example $(-1, 0)$ is moved (if $r\geq 2$), but the net rotation is 0. In particular, $a^{-1}b$ represents one side of a regular $n$-gon of circumradius 2, as shown in Figure~\ref{fig:CRT-shrinking}. $a^{-2}b^2$ is another translation, and so on. We can generate translations from any one vertex of this $n$-gon to any other, as long as $r$ allows the point to remain within both disks, by composing these sequences and their inverses.

\begin{figure}[h!tbp]
	\centering
	\includegraphics[width=.8\textwidth]{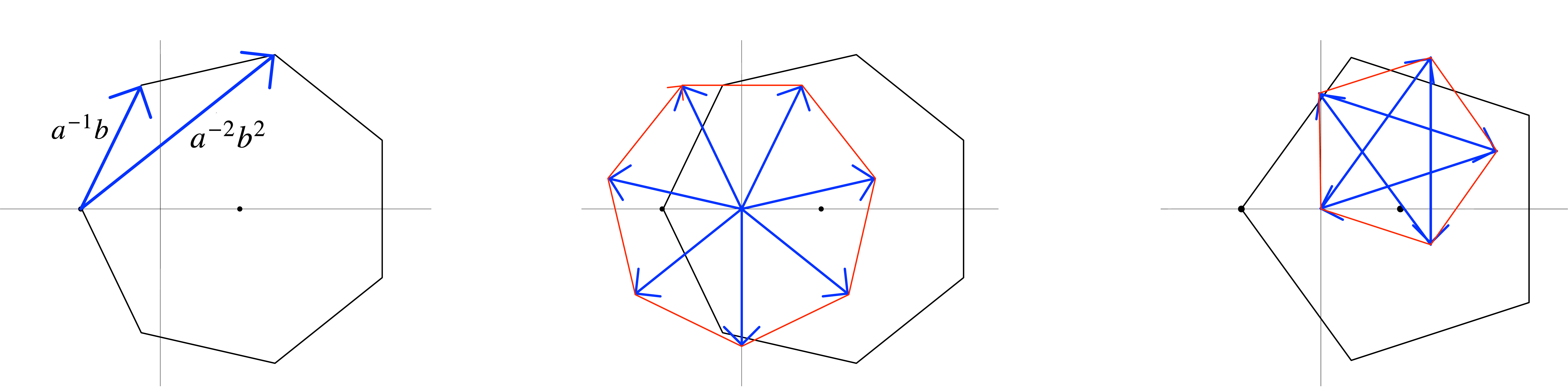}
    \caption{Constructible translations can shrink arbitrarily.}
	\label{fig:CRT-shrinking}
\end{figure}

There are two cases. If $n>5$, we take all the translations from one polygon vertex to an adjacent one, and observe the images we get of the origin under all of them. The resulting $n$ points form another $n$-gon, but smaller than the original. But again, we can form translations between any of these new points by taking differences of the appropriate translations, and then repeat the process, resulting in a yet smaller polygon. Thus, we can generate arbitrarily small translations, and the group must be infinite.

If $n=5$, we start at the origin and apply every other pentagon edge translation, resulting in a pentagram shape whose vertices again form a pentagon smaller than the original. Again, we can iterate.

We can do this with bounded $r$---inspection shows that $r\geq 4$ is sufficient; no moves described ever move a relevant point more than a distance of 4 from either disk center.

If $n_1\neq n_2$, then it is easy to show that $(a^{-1}b)^\alpha$ and $(ba^{-1})^\alpha$, for some $\alpha$, are rotations by $2\pi/\lcm(n_1, n_2)$ about two different centers. We can use these rotations in place of $a$ and $b$ in the proof above (with $r\geq 8$).
\end{proof}

\section*{Geometric Constructions}
For some $n$, we have geometric constructions showing that $GG_n$ is infinite at a value of $r$ matching our numerical estimates for critical radius. 
For other $n$ we have plausible geometric constructions which agree well with our numerical estimates. But for most $n$, all we have is our numerical estimates.

The simplest case is $n=5$. Figure~\ref{fig:n=5} shows the relevant geometry. Using  the shown relationships, simple trigonometry yields
$r=\sqrt{3 + \varphi}\approx 2.149,$
where $\varphi$ is the golden ratio. A similar analysis of the geometry in Figure~\ref{fig:n=10}, where $n=10$, gives $r=\sqrt{4 - \varphi}\approx 1.543$.

The dynamical processes behind the transition from finite to infinite at the critical radius remain mysterious in most cases. However, for $n=5$, we can see what is going on. (We omit due to space a proof that the single generator $ab^{-1}$ produces the same behavior.)

\begin{theorem}
$GG_5$ is infinite at $r=\sqrt{3 + \varphi}$.
\end{theorem}

\begin{proof}
Referring to Figure~\ref{fig:n=5}, interpreted now as the complex plane, let $\zeta_n=e^{2\pi i/n}$, and the point $E=\zeta_5-\zeta_5^2$. Note that $|E+1|=r$.
We focus on how the line segment $E'E$ moves under specific sequences.  The point $F=1-\zeta_5+\zeta_5^2-\zeta_5^3$
lies on $E'E$, as does the point $G=2F-E$.  We have three cases:
\begin{enumerate}
    \item Line segment $E' F'$ is transformed by $a^{-2}  b^{-1} a^{-1} b^{-1}$ to line segment $G F$.
    \item Line segment $F' G'$ is transformed by $a b a b^2$ to line segment $F E$.
    \item Line segment $G' E$ is transformed by $a b a b^{-1} a^{-1} b^{-1}$ to line segment $E' G$.
\end{enumerate}
Together, these three operations can translate any portion of the line segment $E' E$ piecewise onto itself.
At no time does any point leave the intersection of the two disks during these transformations.
The first two cases are translations of length $|F-F'|$, and the third case is a translation of length $|E-G|$.  These two values are not rationally related to the total length $|E - E'|$, since $|E-E'|/|F-F'|=\varphi$.
We can thus map the origin to successive points along $E' E$, by repeatedly choosing the transformation matching the region the point is in, indefinitely; it has an infinite image.
\end{proof}

\begin{figure}[!ht]
\centering
    \begin{subfigure}{0.49\linewidth}
        \includegraphics[width=\linewidth]{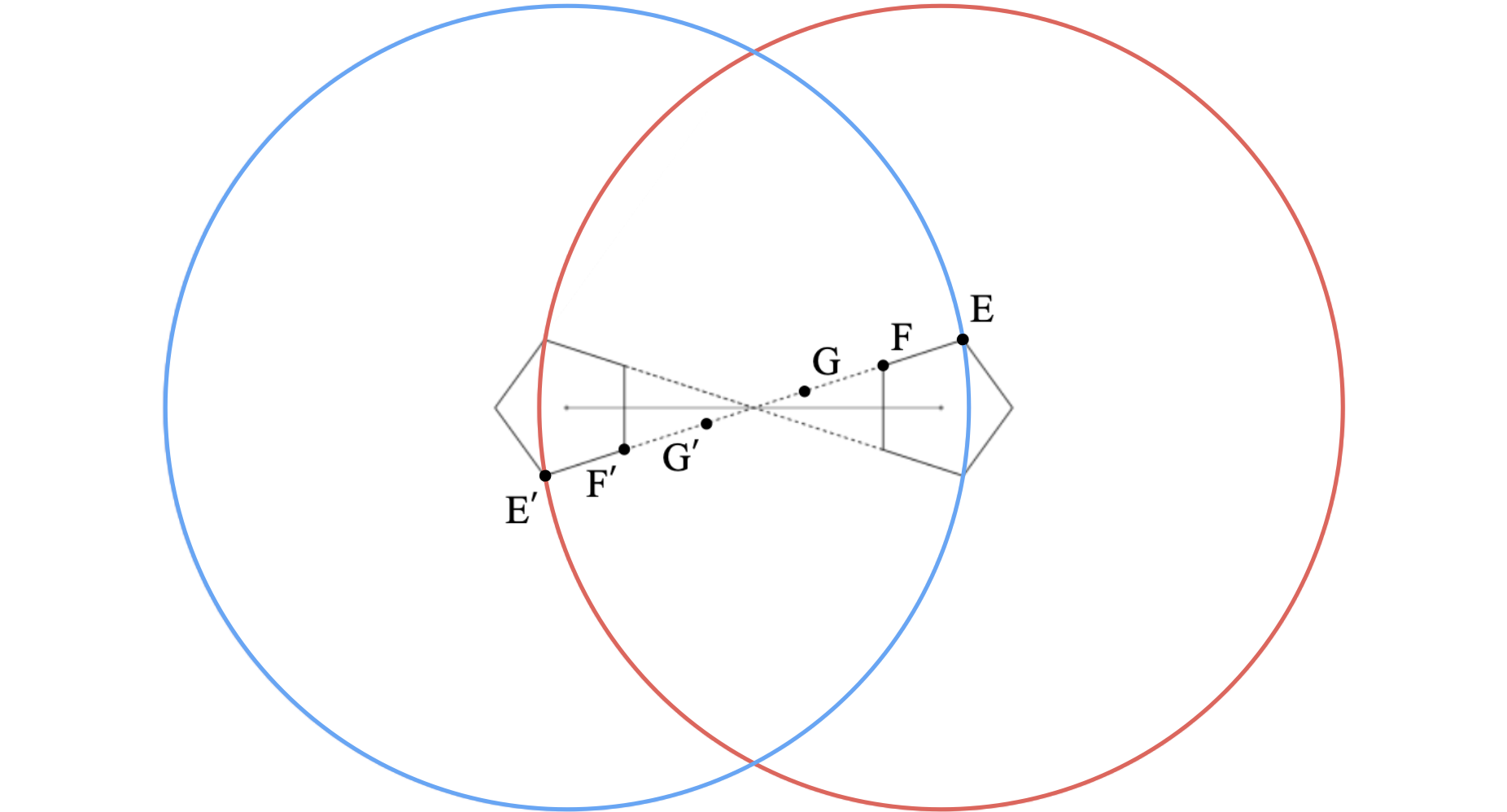}
    
    \caption{}
    \label{fig:n=5}
    \end{subfigure}
\hfil
    \begin{subfigure}{0.49\linewidth}
        \includegraphics[width=\linewidth]{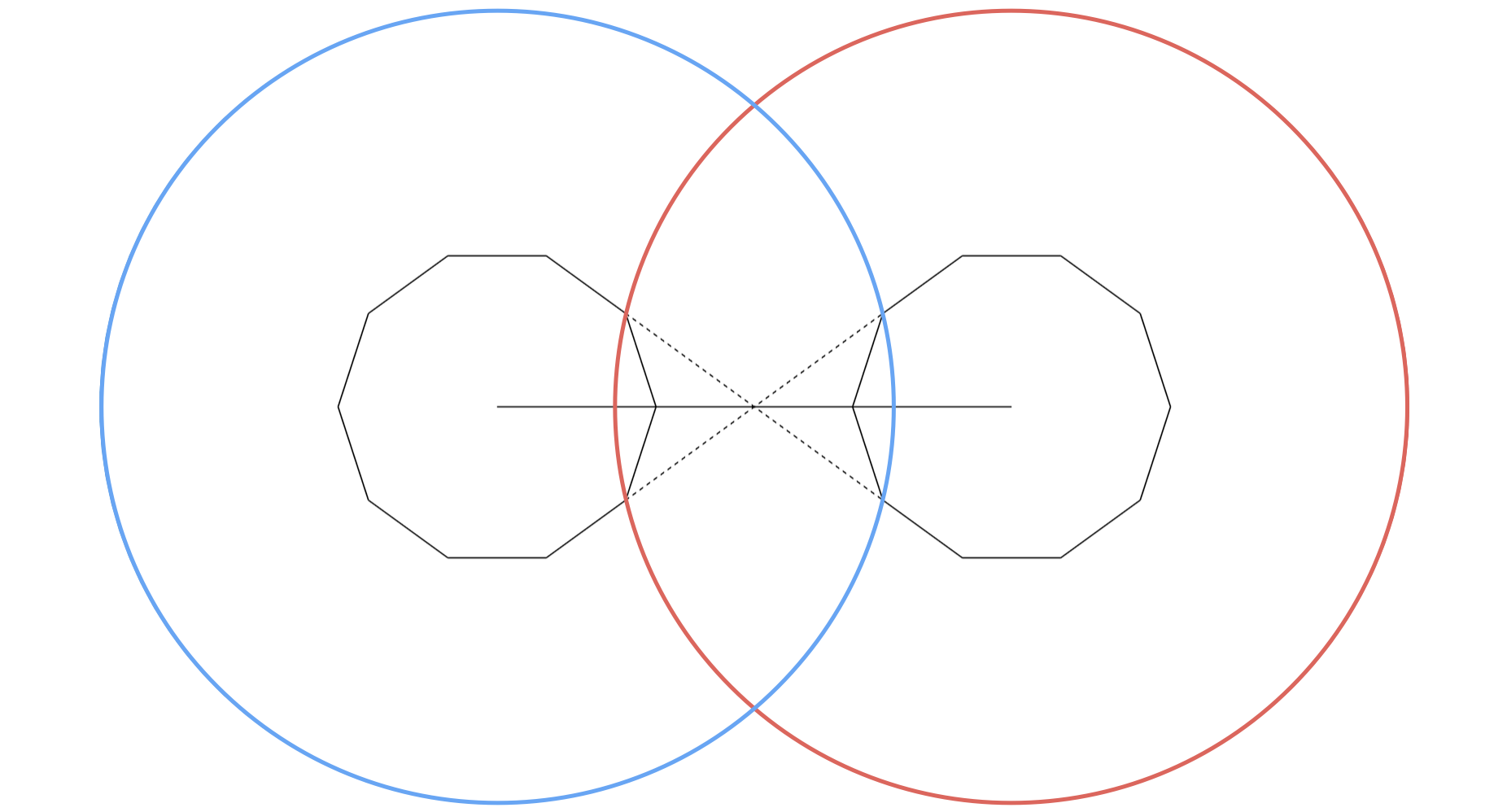}
    
    \caption{}
    \label{fig:n=10}
    \end{subfigure}

    \begin{subfigure}{0.49\linewidth}
        \includegraphics[width=\linewidth]{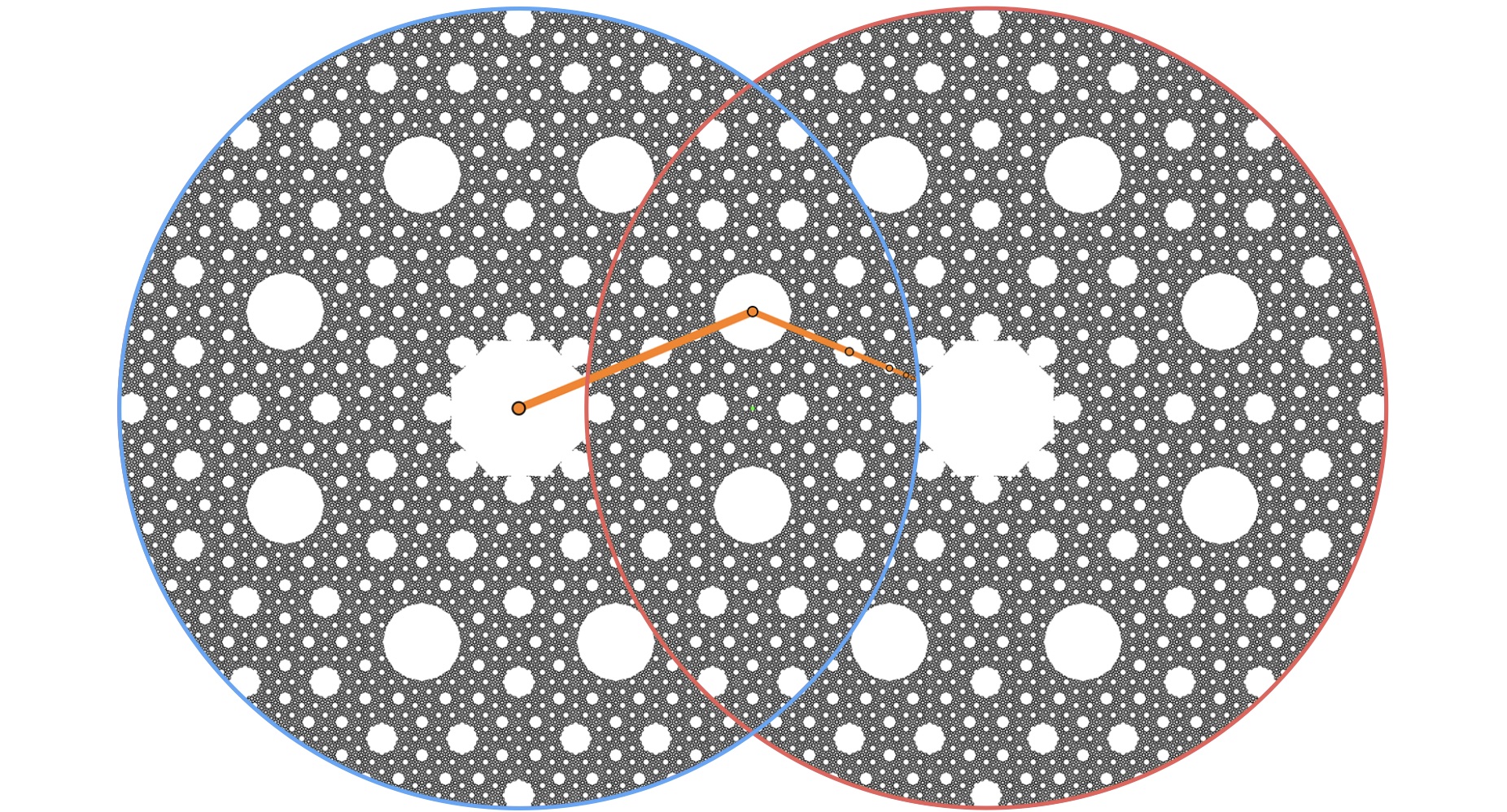}
    
    \caption{}
    \label{fig:n=8}
    \end{subfigure}
\hfil
    \begin{subfigure}{0.49\linewidth}
        \includegraphics[width=\linewidth]{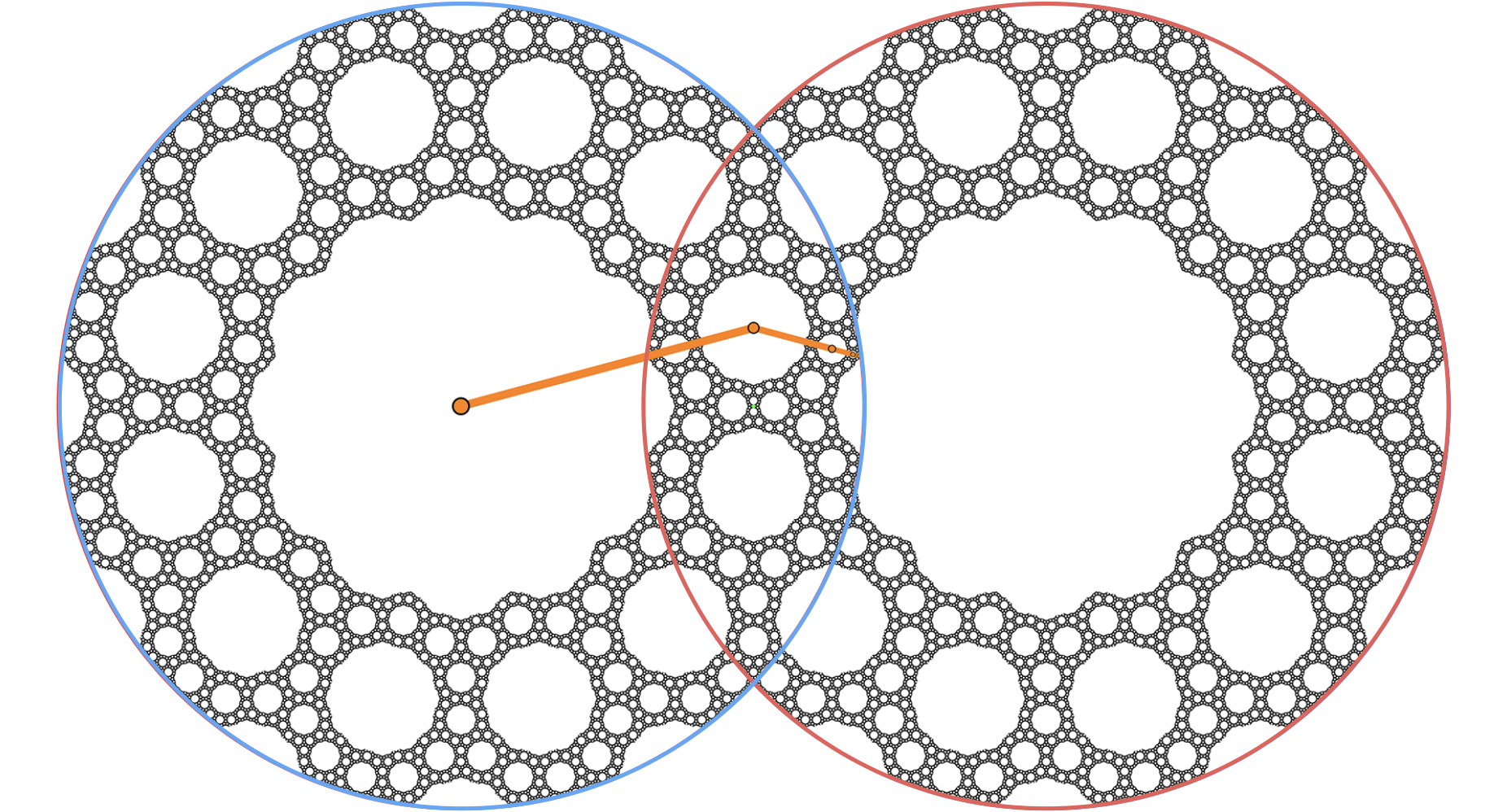}
    \caption{}
    \label{fig:n=12}
    \end{subfigure}

\caption{Geometric constructions for critical radius: (a) $n=5$, (b) $n=10$, (c) $n=8$, (d) $n=12$.}
    \label{fig:my figure}
    \end{figure}


For the cases of $n=8$ (Figure~\ref{fig:n=8}) and $n=12$ (Figure~\ref{fig:n=12}), their characteristic fractals (see below) can provide insight.
A path of consecutive line segments that follows the fractal structure can be realized, starting from the center of one disk and approaching the disk boundary from the interior. For $n=8$, consecutive segments scale down by a factor of $\sqrt{2}-1$ and traverse angles of $\pi/8$. We can calculate the corresponding limit point from the path to yield $r=\sqrt{5(2-\sqrt{2})}\approx1.711.$
For $n=12$, an analogous construction gives a scale factor of $2-\sqrt{3}$, and $r=\sqrt{2(20-11\sqrt{3})}\approx1.377.$
These closed-form radii rely on the limit point lying on the disk boundary at the critical value, which has not been proven.

\section*{Critical Transitions and Fractals}

Precisely at any group's critical radius, we always observe a distinct fractal embedded in the image.
It seems remarkable that these fractals have gone unnoticed for so long; they seem as natural as, e.g., the Mandelbrot set. 
We can define the \emph{characteristic fractal} for $GG$ to be the set of points with infinite orbits at $GG(r_c(GG))$. In particular, this associates a unique fractal with every $n \not\in \{2, 3, 4, 6\}$, and similarly when $n_1\neq n_2$.\footnote{We also have fractals when $r_1\neq r_2$, but in that case the critical radius becomes a two-parameter family.} The canonical example is the fractal for $n=5$, shown in Figure~\ref{fig:n=5-fractal}. We also include in this paper the fractals for $n=8$ (Figure~\ref{fig:n=8}) and $n=12$ (Figure~\ref{fig:n=12}). In Appendix B, we include higher-resolution images of the characteristic fractals for all $n$ up to 20.

The appearance of fractals here seems somewhat mysterious. The hallmark of a fractal is the repeating of a pattern on successively smaller scales. But unlike fractals constructed with an explicit recursive rule, there are no scaling operations in these systems, only rotations. Where do they come from? Why do they look so different for different $n$? Some insight may be gained by considering Figure~\ref{fig:CRT-shrinking}, which shows that in fact, rotations can combine to shrink patterns.

A natural question for any characteristic fractal is whether it is the closure of the orbit of a single point, or whether it consists of finitely or infinitely many disjoint closures of orbits. In most cases the fractals seem to be closures of the orbit of a single point. For $n=23$, it appears that the fractal consists of two distinct pieces, symmetric about the $x$-axis. However, it is possible that these pieces are connected, and we have not been able to reach this connection numerically.

An especially interesting case is $n=7$. Here, there seems to be a clear fractal structure. We can zoom in several times and see the same structure repeating. But then at a certain point, after zooming in by a factor of 500,000, the pattern suddenly changes. This was not apparent until we imaged the fractal to over 2 trillion points. It is possible that this broken scale symmetry is a numerical artifact, but evidence argues against this. If it is real, this is perhaps the most mysterious process related to these systems we have yet observed. 
A video showing this transition may be found at \url{https://www.youtube.com/watch?v=FFeSlh0ifYE}.
A portion of $GG_7$ just short of the critical radius, shown in Figure~\ref{fig:n=7}, reveals the characteristic $n=7$ fractal motif.

In Figure~\ref{fig:Koch}, we see that embedded within $GG_{12}$ lies a Koch-snowflake-like fractal seemingly based on four-fold symmetry, rather than the traditional three-fold.

\begin{figure}[h!tbp]
\centering
\begin{minipage}[t]{0.35\textwidth}
\centering
	\includegraphics[width=\textwidth]{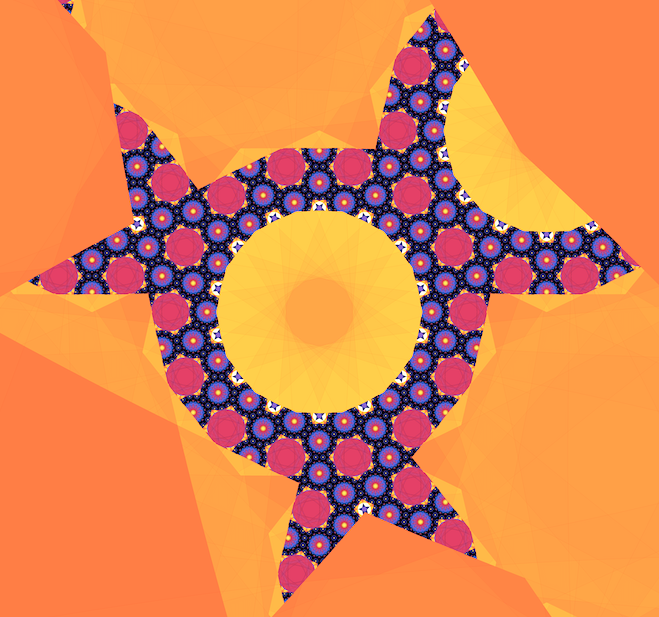}
    \subcaption{A portion of $GG_7(1.6233)$. }
	\label{fig:n=7}
 \end{minipage}
 \hfill
\begin{minipage}[t]{0.57\textwidth} 
\centering
	\includegraphics[width=\textwidth]{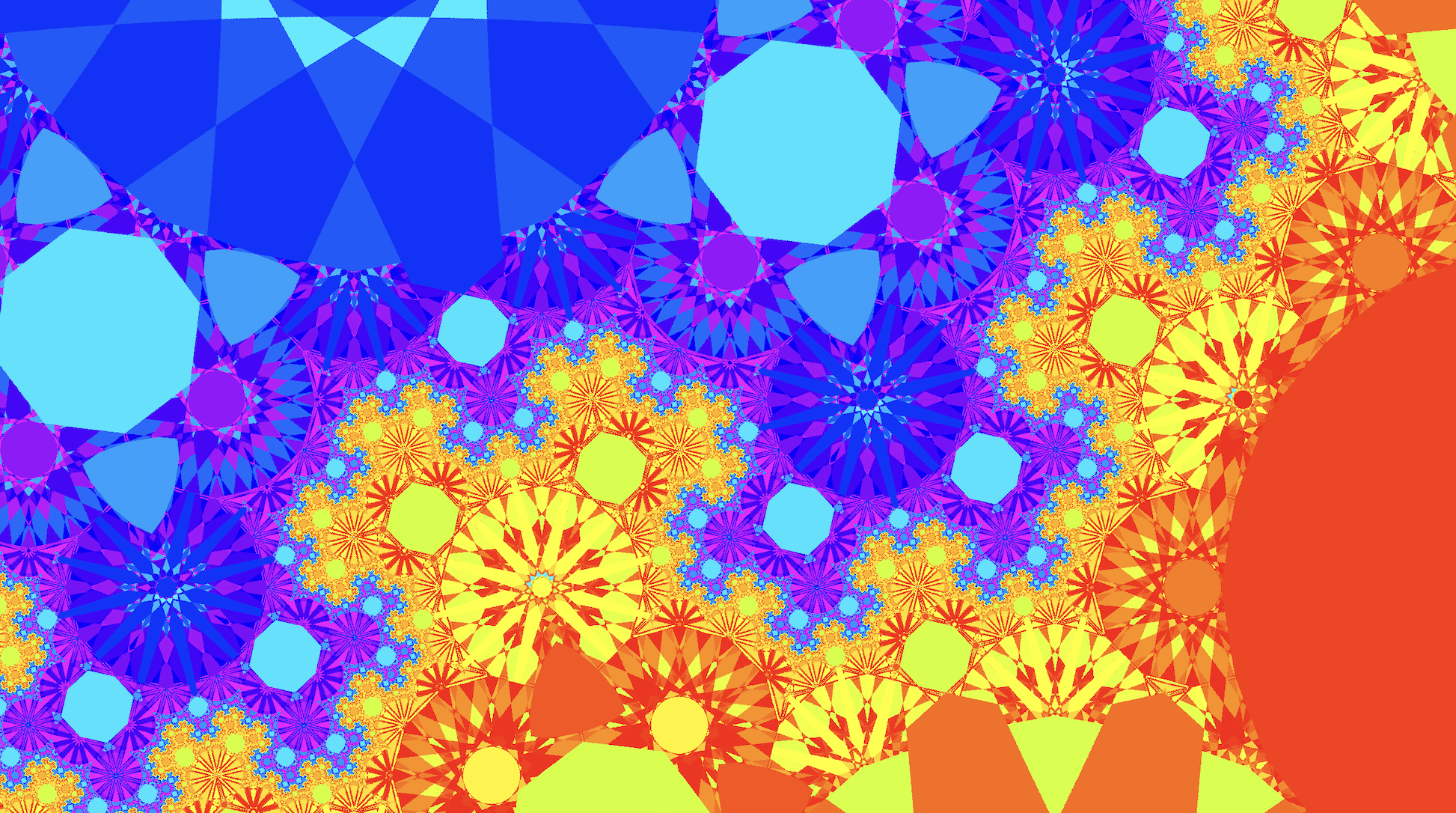}
    \subcaption{A portion of the generator $ab^{-1}$ for $GG_{12}(\sqrt{2})$.
}
	\label{fig:Koch}
 \end{minipage}
    \caption{}
	\label{fig:fractals}
\end{figure}

\section*{Numerical Models and Algorithms}
All of the images in this paper were generated by programs which simulate compound symmetry groups. Conceptually, for each point we want to image, we consider all possible rotation sequences applied to it, and plot the results---i.e., we plot its orbit. Figure~\ref{fig:n=5-fractal}, for example, was generated this way. When we want to show the action of the group on the entire space, it is more efficient to only image the disk boundaries. 
We image a boundary by discretizing it into small segments, imaging those segments, drawing the images into a high-resolution bitmap, then filling the resulting spaces within the bounded regions with appropriate colors, according to some coloring rule.
This is how Figures~\ref{fig:3-5}, \ref{fig:n=7}, \ref{fig:n=9}, and \ref{fig:n=5-pretty} were generated.

\begin{figure}[h!tbp]
\centering
\hspace{.2in}
\begin{minipage}[c]{0.42\textwidth} 
	\centering
	\includegraphics[width=\textwidth]{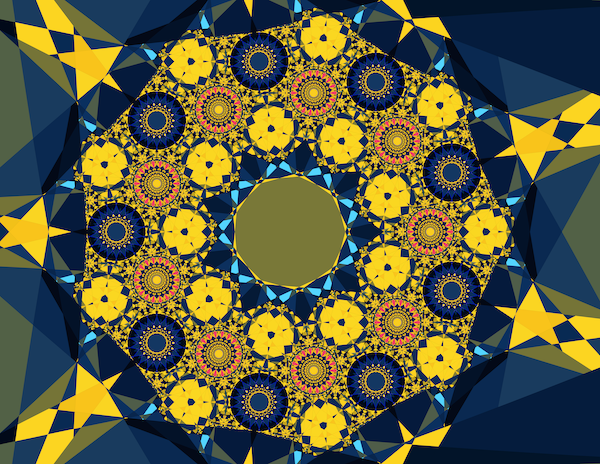}
	\subcaption{}
	\label{fig:n=9}
\end{minipage}
\hfill
\begin{minipage}[c]{0.4\textwidth}
\centering
	\includegraphics[width=\textwidth]{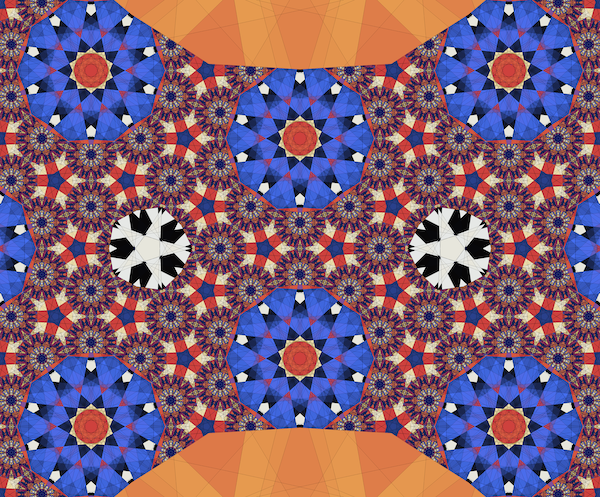}
	\subcaption{}
	\label{fig:n=5-pretty}
\end{minipage}
\hspace{.2in}
\caption{Filled images: (a) a portion of $GG_9(1.408)$, (b) a portion of $GG_5(2.144)$.}
\end{figure}

The use of frontier search \cite{Korf2005} when generating orbits helps enormously, allowing us to not keep the entire orbit in memory. It has enabled some points to be imaged to over 10 trillion distinct targets, helping us refine our critical radius estimates, and exposing the curious behaviors of some fractals discussed above.

\section*{Single-Generator Images}

Additional structure can be revealed by imaging orbits under a single generator, rather than all group elements. In the simplest case, we have $a^{\alpha}b^{\beta}$. Two examples are shown in Figure~\ref{fig:single-generator}. This defines a discrete dynamical system which is an iterated piecewise isometry. Some work has been done to characterize these systems, but many questions remain. \cite{Smith2019, Gossow2020}

\begin{figure}[h!tbp]
\centering
\begin{minipage}[t]{0.35\textwidth} 
	\includegraphics[width=\textwidth]{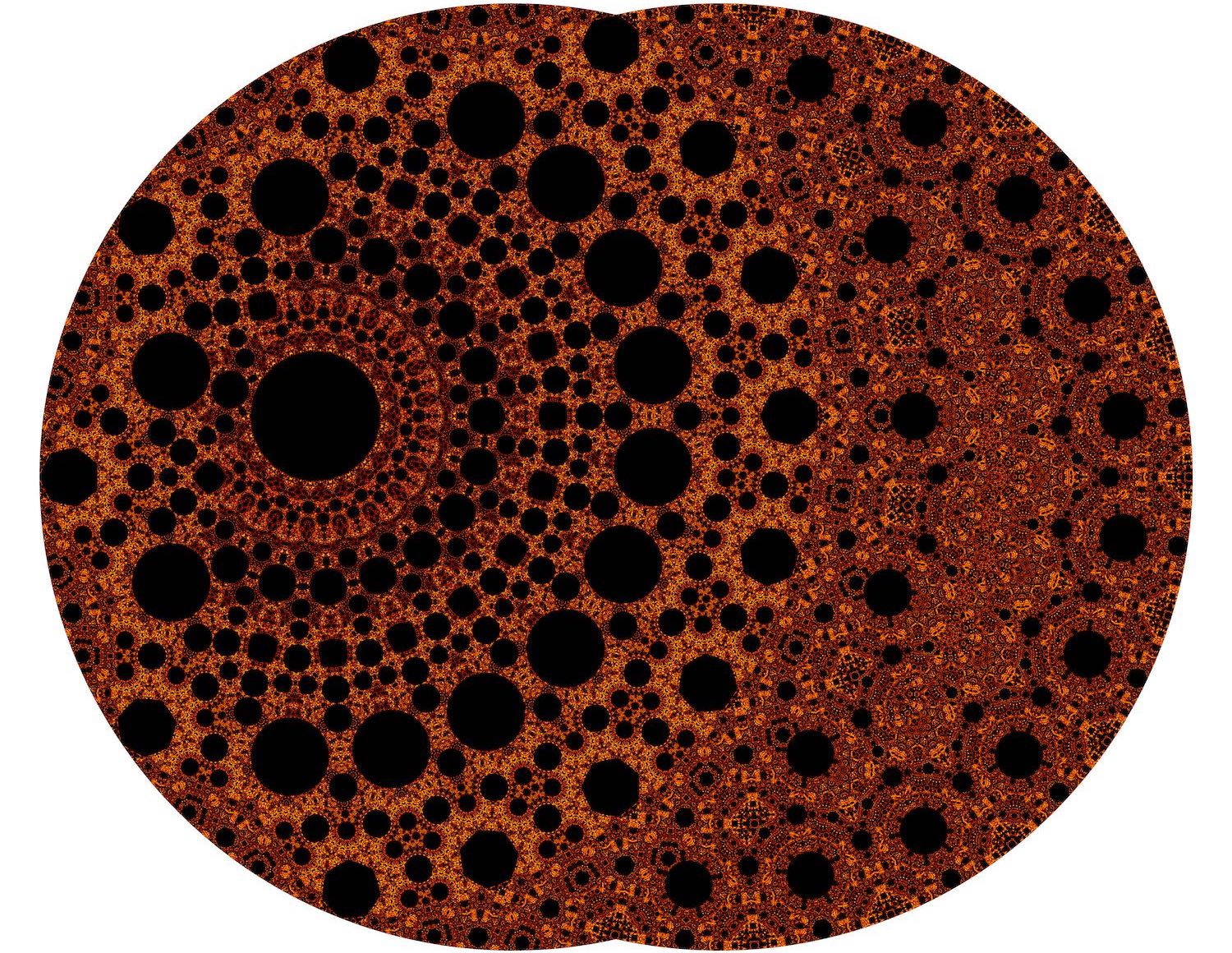}
        	\subcaption{ ${GG_{21}}(5.82)$ using $a^{2}b^{-1}$} 
        	\label{fig:n=21}
\end{minipage}
~ 
\begin{minipage}[t]{0.4\textwidth} 
	\includegraphics[width=\textwidth]{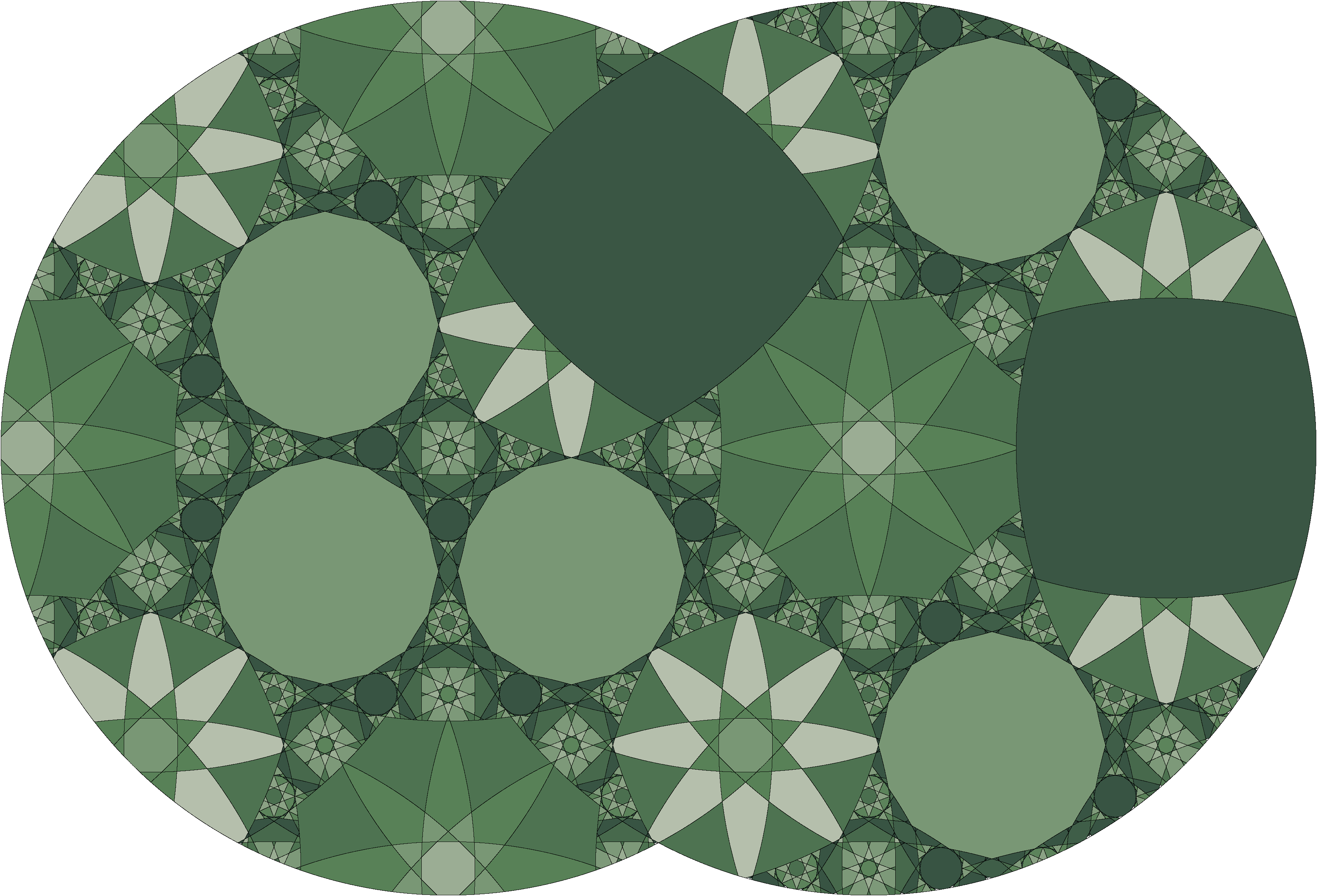}
        	\subcaption{${GG_{8}}(2.12)$ using $a^{2}b^{5}$}
        	\label{fig:n=16}
\end{minipage}
\caption{(a) The orbit of the upper intersection point is plotted and colored according to its local density. (b) The full boundary is plotted, with spaces colored according to the order of their orbit.
}
\label{fig:single-generator}
\end{figure}

\section*{Critical Radii}

Table~\ref{table:radii} summarizes our knowledge of critical radii for $GG_n$ with $n < 20$. A more complete table, up to $n=100$, may be found in Appendix A. In all cases, points were found with a minimum of 10 billion images, with some up to 10 trillion. These estimates may be too high because points with infinite image were missed, or too low because the points searched had very large but finite images. However, there is good agreement with the geometrically derived values, and we have confidence they are good to about 5 decimal places.

\begin{table}[h!tbp]
\centering
\caption{Critical radii for $GG_n$.}
{\scriptsize 
\begin{minipage}[t]{0.7\textwidth} 
\renewcommand{\arraystretch}{1.9}
\begin{tabular}{|c|c|c|c|}
\hline
  $n$ & Numerical estimate & Algebraic expression & Minimum Polynomial  \\
\hline
5 & 2.148961 & $\sqrt{3 + \varphi}\approx 2.148961$ & $x^4 - 7x^2 + 11$ \\
\hline
7 & 1.623574 & & \\
\hline
8 & 1.711411 & $\sqrt{5(2-\sqrt{2})} \approx 1.711412$ & $x^4 - 20x^2 + 50$ \\
\hline
9 & 1.408482 & & \\
\hline
10 & 1.543357 & $\sqrt{4 - \varphi}\approx 1.543362$ & $x^4 - 7x^2 + 11$\\
\hline
11 & 1.290582 & & \\
\hline
12 & 1.376547 & $\sqrt{2(20-11\sqrt{3})} \approx 1.376547$ & $x^4 - 80x^2 + 148$\\
\hline
\end{tabular}
\end{minipage}
\renewcommand{\arraystretch}{1}
\begin{minipage}[t]{0.29\textwidth} 
\renewcommand{\arraystretch}{1.9}
\begin{tabular}{|c|c|}
\hline
  $n$ & Numerical estimate     \\
\hline
13 & 1.213594   \\
\hline
14 & 1.196554  \\
\hline
15 & 1.163276  \\
\hline
16 & 1.148470  \\
\hline
17 & 1.127509  \\
\hline
18 & 1.121505  \\
\hline
19 & 1.104246  \\
\hline
\end{tabular}
\end{minipage}
\renewcommand{\arraystretch}{1}
}
\label{table:radii}
\end{table}


\section*{Summary and Conclusions}

The concept of compound symmetry group opens up a new frontier in mathematics. Here we have just begun this exploration, by considering the two-disk compound symmetry groups. This investigation has revealed a new family of fractals, and a rich new source of spaces that combine symmetries in interesting and unexpected ways. 
These are new ``places'' to explore, similar for example to ``Seahorse Valley'' in the Mandelbrot set, that have until now remain unseen. We must omit due to space many additional topics we would like to cover, such as the appearance of quasicrystals when we move beyond the critical radius, and other observations of aperiodic behavior.

While we have made significant progress in understanding two-disk compound symmetry groups, yet more work can still be done. In many cases, we lack even a basic theoretical understanding of the behaviors that cause the transition from finite to infinite size in these groups. For example, does every infinite two-disk group contain a point whose image is infinite, or a generator of infinite order? Is the critical radius of a two-disk system always an algebraic number, and is there a general formula for the critical radius? What dynamics are responsible for the creation of the fractals?

Similar questions can be raised more generally for multi-disk systems or arbitrary compound symmetry groups, where the behavior governing infinite size groups is more complicated. For example, we observe that Theorem~\ref{thm:not_2346_group_infinite} fails to generalize to three-disk systems: consider a set of three disks centered at, say, $(0,0)$, $(1,0)$, and $(\sqrt{2},0)$, and consider the compound symmetry group obtained by taking rotation increments $n_1 = n_2 = n_3 = 2$. Then if we choose disk radii to be sufficiently large, the resulting compound symmetry group will be infinite, because the three corresponding rotations of the plane generate a pair of translations along the $x$-axis whose ratio is irrational. Another question that presents itself is whether it is even decidable whether a given multi-disk compound symmetry group is finite.

Some movies of systems with varying zooms and radii may be found at \url{https://tinyurl.com/yz9pntwy}.


\section*{Acknowledgments}
We thank Doug Engel for Gizmo Gears, and his two definitive books on circle puzzles. We thank Brandon Enright for his significant contributions, insights, and ideas. We thank Oskar van Deventer and Carl Hoff, who first introduced Hearn to this problem, and contributed insights. We also thank Bram Cohen, Scott Elliott, Landon Kryger, Andreas Nortmann, Jason Smith, Nathaniel Virgo, and all others who shared interest and insights in the \href{https://twistypuzzles.com}{twistypuzzles.com} forum.

    
{\setlength{\baselineskip}{13pt} 
\raggedright				

} 

\section*{Appendix A: Critical Radii up to $n=100$}

\begin{table}[h!tbp]
\centering
\caption{Critical radii for $GG_n$.}
{\scriptsize 
\begin{minipage}[t]{0.24\textwidth} 
\vspace{0pt}
\renewcommand{\arraystretch}{1.7}
\begin{tabular}{|c|c|}
\hline
  $n$ & Numerical estimate  \\
\hline
5 & 2.148961 \\
\hline
7 & 1.623574 \\
\hline
8 & 1.711411 \\
\hline
9 & 1.408482 \\
\hline
10 & 1.543357 \\
\hline
11 & 1.290582 \\
\hline
12 & 1.376547 \\
\hline
13 & 1.213594   \\
\hline
14 & 1.196554  \\
\hline
15 & 1.163276  \\
\hline
16 & 1.148470  \\
\hline
17 & 1.127509  \\
\hline
18 & 1.121505  \\
\hline
19 & 1.104246  \\
\hline
20 & 1.100581   \\
\hline
21 & 1.086016   \\
\hline
22 & 1.078162   \\
\hline
23 & 1.072011   \\
\hline
24 & 1.071404   \\
\hline
25 & 1.061321   \\
\hline
26 & 1.056958   \\
\hline
27 & 1.052668   \\
\hline
28 & 1.049731   \\
\hline
29 & 1.045731   \\
\hline
\end{tabular}
\end{minipage}
\begin{minipage}[t]{0.24\textwidth} 
\vspace{0pt}
\renewcommand{\arraystretch}{1.7}
\begin{tabular}{|c|c|}
\hline
  $n$ & Numerical estimate     \\
\hline
30 & 1.043758   \\
\hline
31 & 1.040158   \\
\hline
32 & 1.037741   \\
\hline
33 & 1.035512   \\
\hline
34 & 1.033498   \\
\hline
35 & 1.031647  \\
\hline
36 & 1.030127  \\
\hline
37 & 1.028368  \\
\hline
38 & 1.026958  \\
\hline
39 & 1.025591  \\
\hline
40 & 1.024352  \\
\hline
41 & 1.023174   \\
\hline
42 & 1.022142   \\
\hline
43 & 1.021105    \\
\hline
44 & 1.020197   \\
\hline
45 & 1.019279   \\
\hline
46 & 1.018459   \\
\hline
47 & 1.017692   \\
\hline
48 & 1.016968   \\
\hline
49 & 1.016287   \\
\hline
50 & 1.015651   \\
\hline
51 & 1.015053   \\
\hline
52 & 1.014492   \\
\hline
53 & 1.013942   \\
\hline
\end{tabular}
\end{minipage}
\begin{minipage}[t]{0.24\textwidth} 
\vspace{0pt}
\renewcommand{\arraystretch}{1.7}
\begin{tabular}{|c|c|}
\hline
  $n$ & Numerical estimate     \\
\hline
54 & 1.013437   \\
\hline
55 & 1.012954   \\
\hline
56 & 1.012498   \\
\hline
57 & 1.012067   \\
\hline
58 & 1.011657   \\
\hline
59 & 1.011267  \\
\hline
60 & 1.010908  \\
\hline
61 & 1.010545  \\
\hline
62 & 1.010210  \\
\hline
63 & 1.009890  \\
\hline
64 & 1.009585  \\
\hline
65 & 1.009294   \\
\hline
66 & 1.009018   \\
\hline
67 & 1.008750   \\
\hline
68 & 1.008499   \\
\hline
69 & 1.008253   \\
\hline
70 & 1.008020   \\
\hline
71 & 1.007797   \\
\hline
72 & 1.007582   \\
\hline
73 & 1.007377   \\
\hline
74 & 1.007180   \\
\hline
75 & 1.006990   \\
\hline
76 & 1.006809   \\
\hline
77 & 1.006633   \\
\hline
\end{tabular}
\end{minipage}
\begin{minipage}[t]{0.24\textwidth} 
\vspace{0pt}
\renewcommand{\arraystretch}{1.7}
\begin{tabular}{|c|c|}
\hline
  $n$ & Numerical estimate     \\
\hline
78 & 1.006465   \\
\hline
79 & 1.006302   \\
\hline
80 & 1.006146   \\
\hline
81 & 1.005996   \\
\hline
82 & 1.005851   \\
\hline
83 & 1.005712  \\
\hline
84 & 1.005578  \\
\hline
85 & 1.005447  \\
\hline
86 & 1.005321  \\
\hline
87 & 1.005200  \\
\hline
88 & 1.005083  \\
\hline
89 & 1.004970   \\
\hline
90 & 1.004860   \\
\hline
91 & 1.004754   \\
\hline
92 & 1.004652   \\
\hline
93 & 1.004552   \\
\hline
94 & 1.004456   \\
\hline
95 & 1.004363   \\
\hline
96 & 1.004273   \\
\hline
97 & 1.004186   \\
\hline
98 & 1.004101   \\
\hline
99 & 1.004019   \\
\hline
100 & 1.003939  \\
\hline
\end{tabular}
\end{minipage}
\renewcommand{\arraystretch}{1}
}
\label{table:more-radii}
\end{table}

\pagebreak

\section*{Appendix B: Characteristic Fractals up to $n=20$}

Here we show only the set of points with infinite orbit, at the critical radius, for each $n$.

\begin{figure}[h!tbp]
	\centering
	\includegraphics[width=\textwidth]{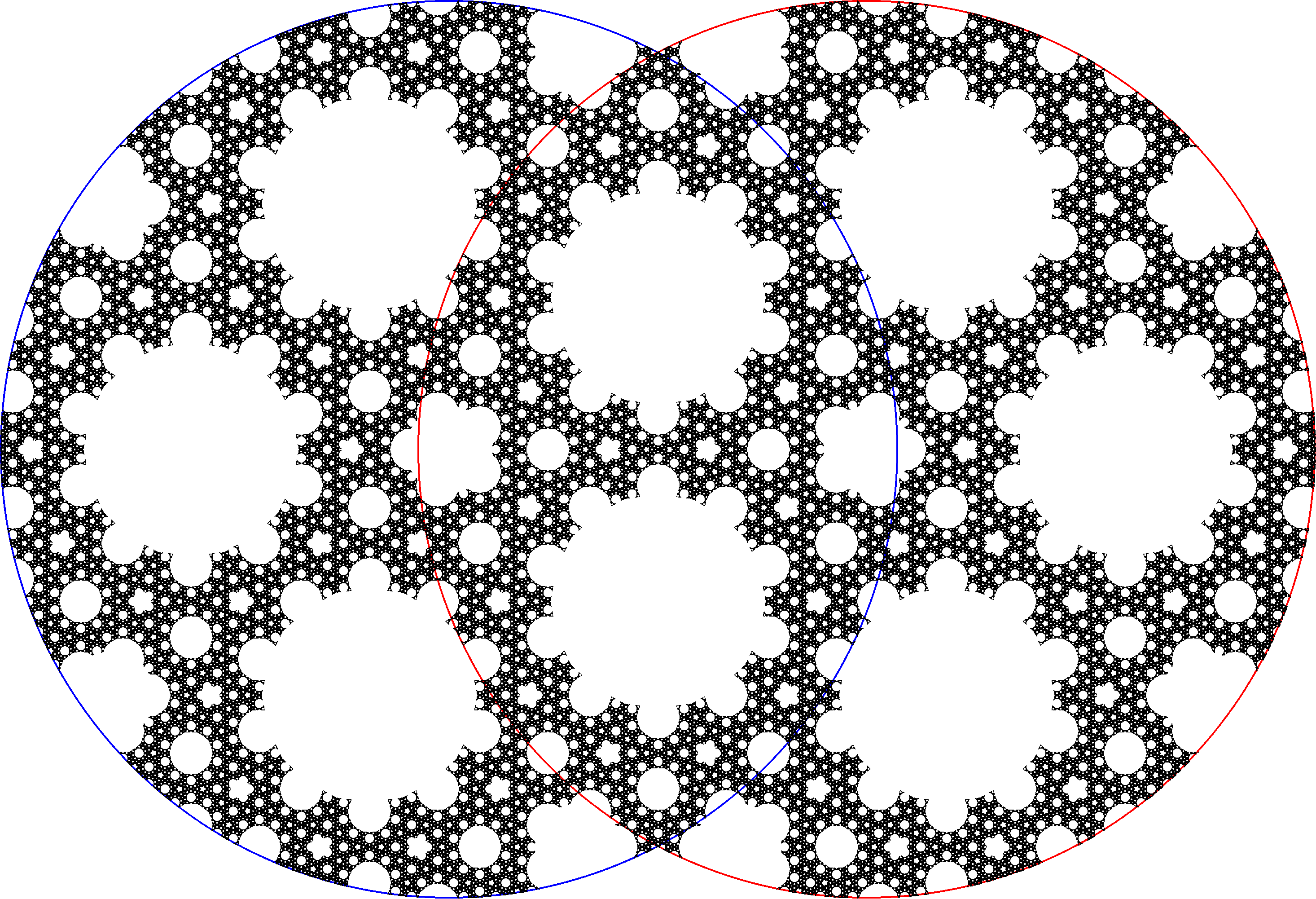}
    \caption{$GG_5(2.148958)$.}
\end{figure}

\begin{figure}[h!tbp]
	\centering
	\includegraphics[width=\textwidth]{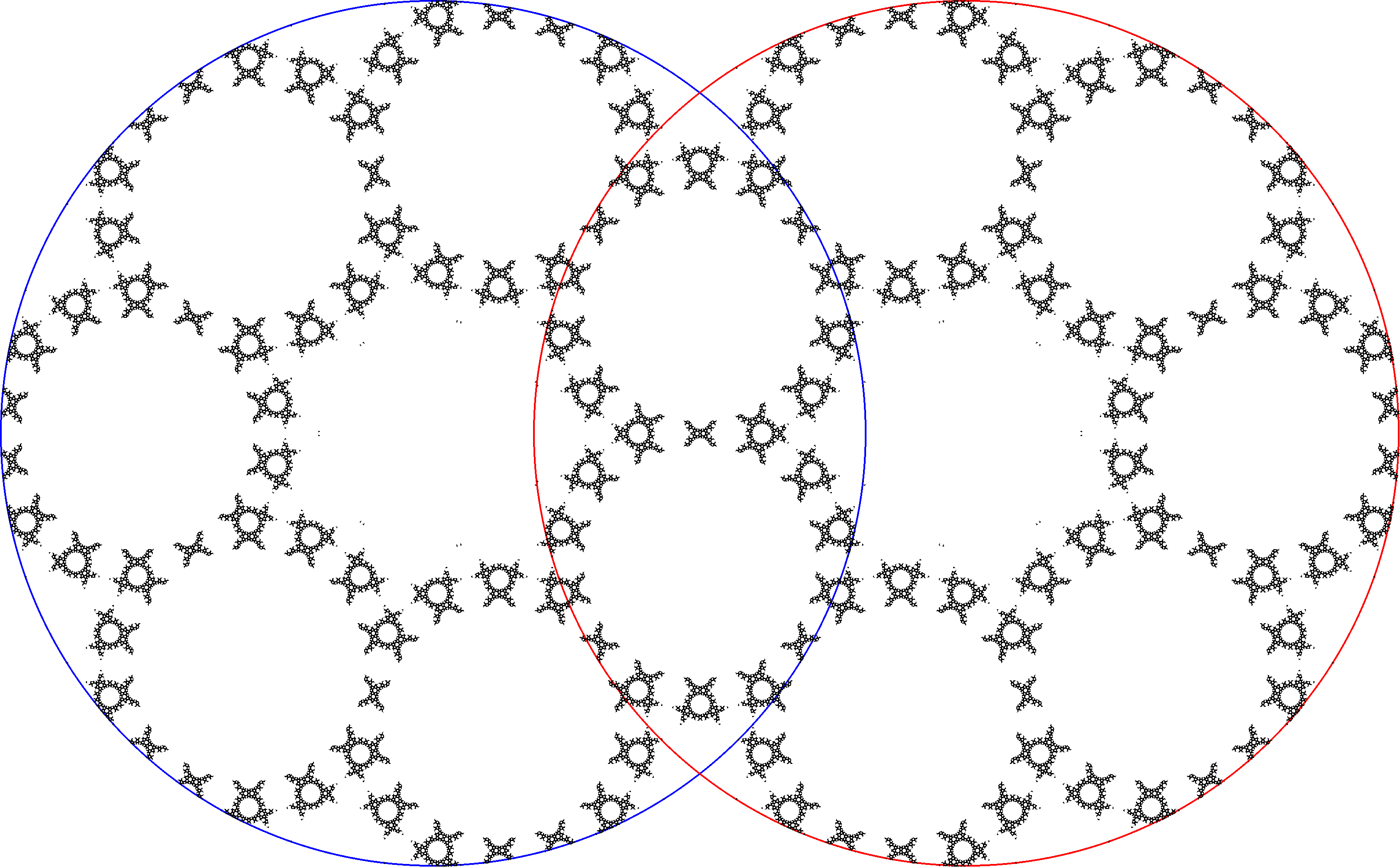}
    \caption{$GG_7(1.623574)$.}
\end{figure}

\begin{figure}[h!tbp]
	\centering
	\includegraphics[width=\textwidth]{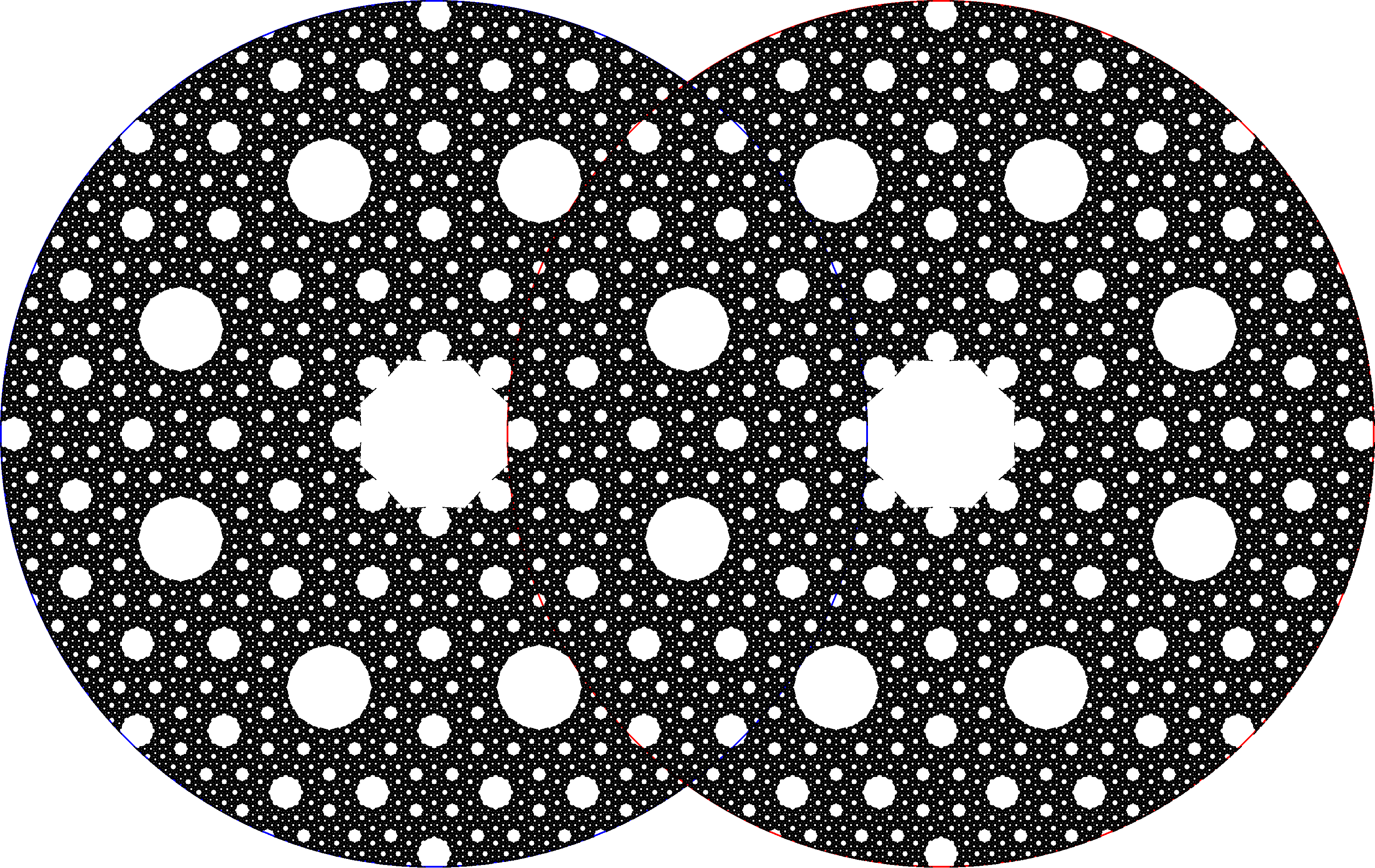}
    \caption{$GG_8(1.711405)$.}
\end{figure}

\begin{figure}[h!tbp]
	\centering
	\includegraphics[width=\textwidth]{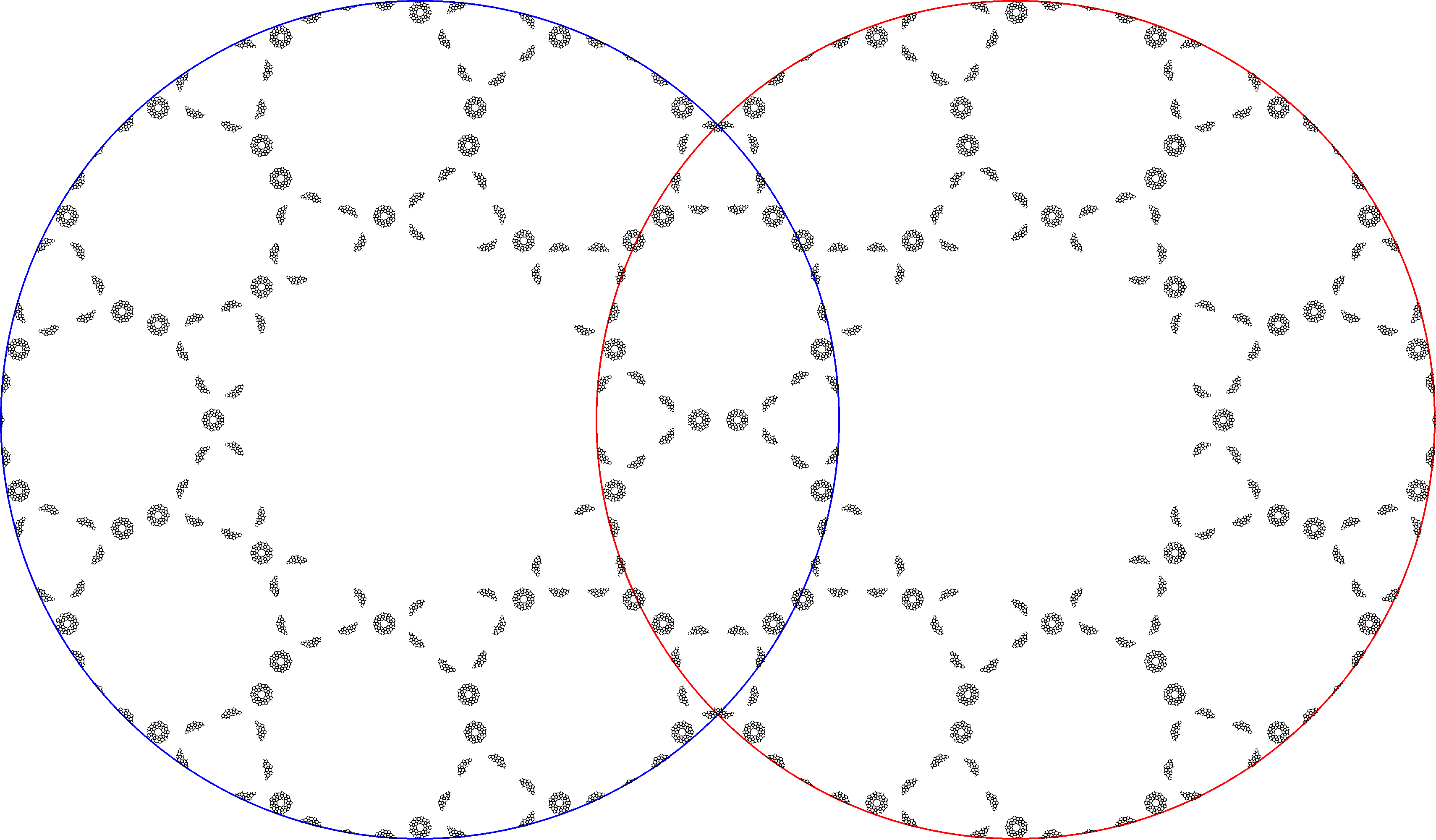}
    \caption{$GG_9(1.408482)$.}
\end{figure}

\begin{figure}[h!tbp]
	\centering
	\includegraphics[width=\textwidth]{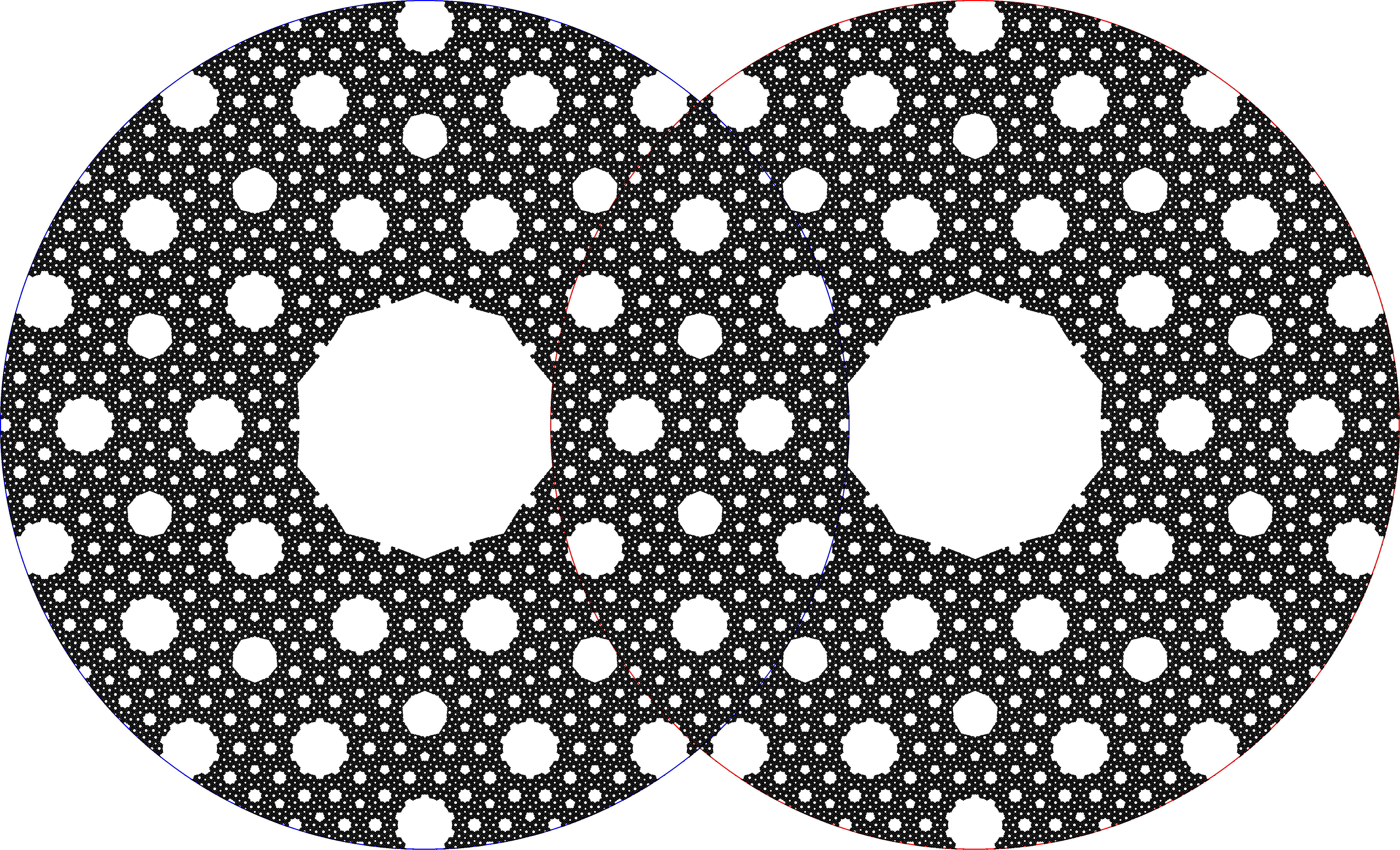}
    \caption{$GG_{10}(1.543357)$.}
\end{figure}

\begin{figure}[h!tbp]
	\centering
	\includegraphics[width=\textwidth]{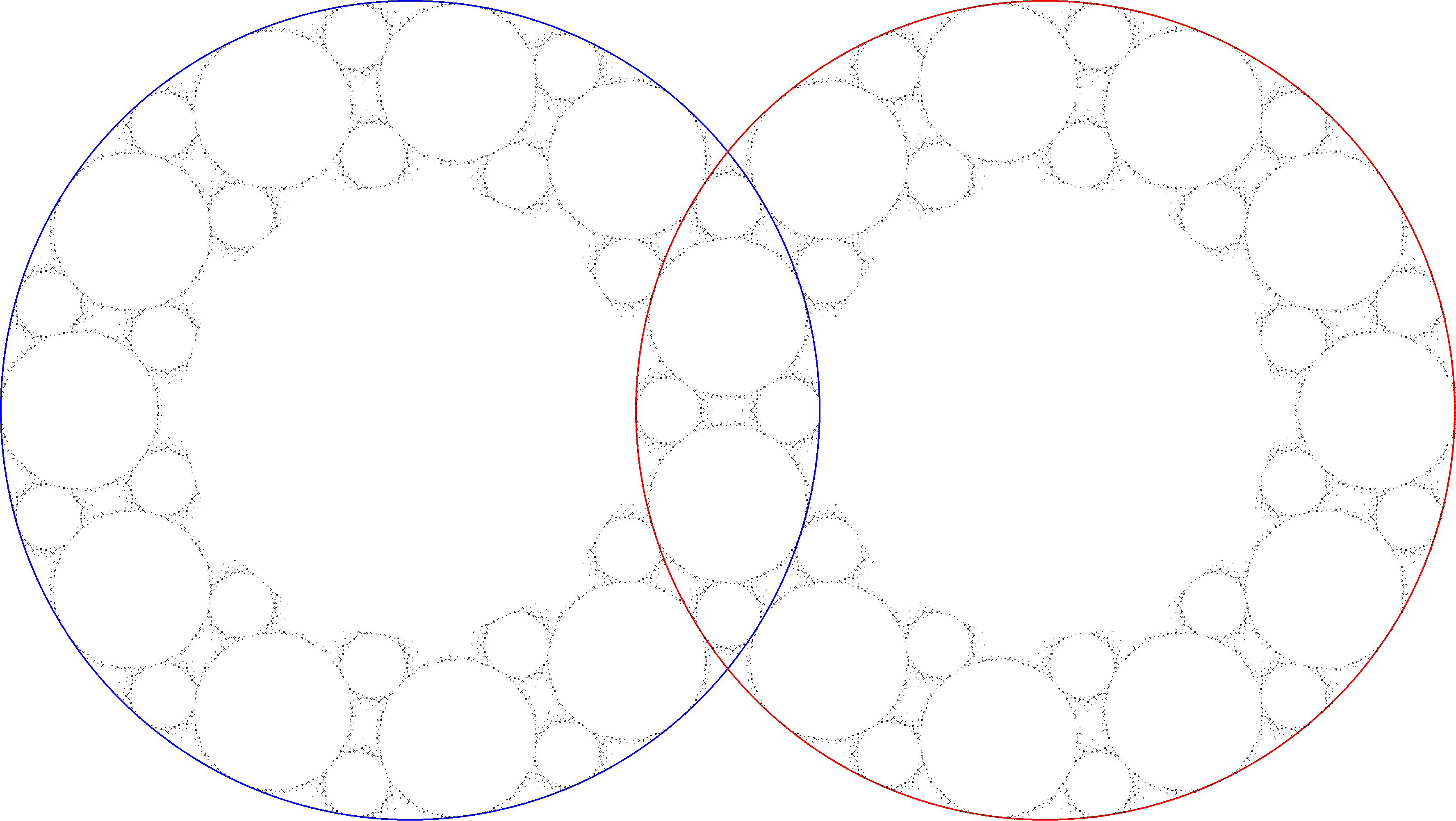}
    \caption{$GG_{11}(1.290582)$.}
\end{figure}

\begin{figure}[h!tbp]
	\centering
	\includegraphics[width=\textwidth]{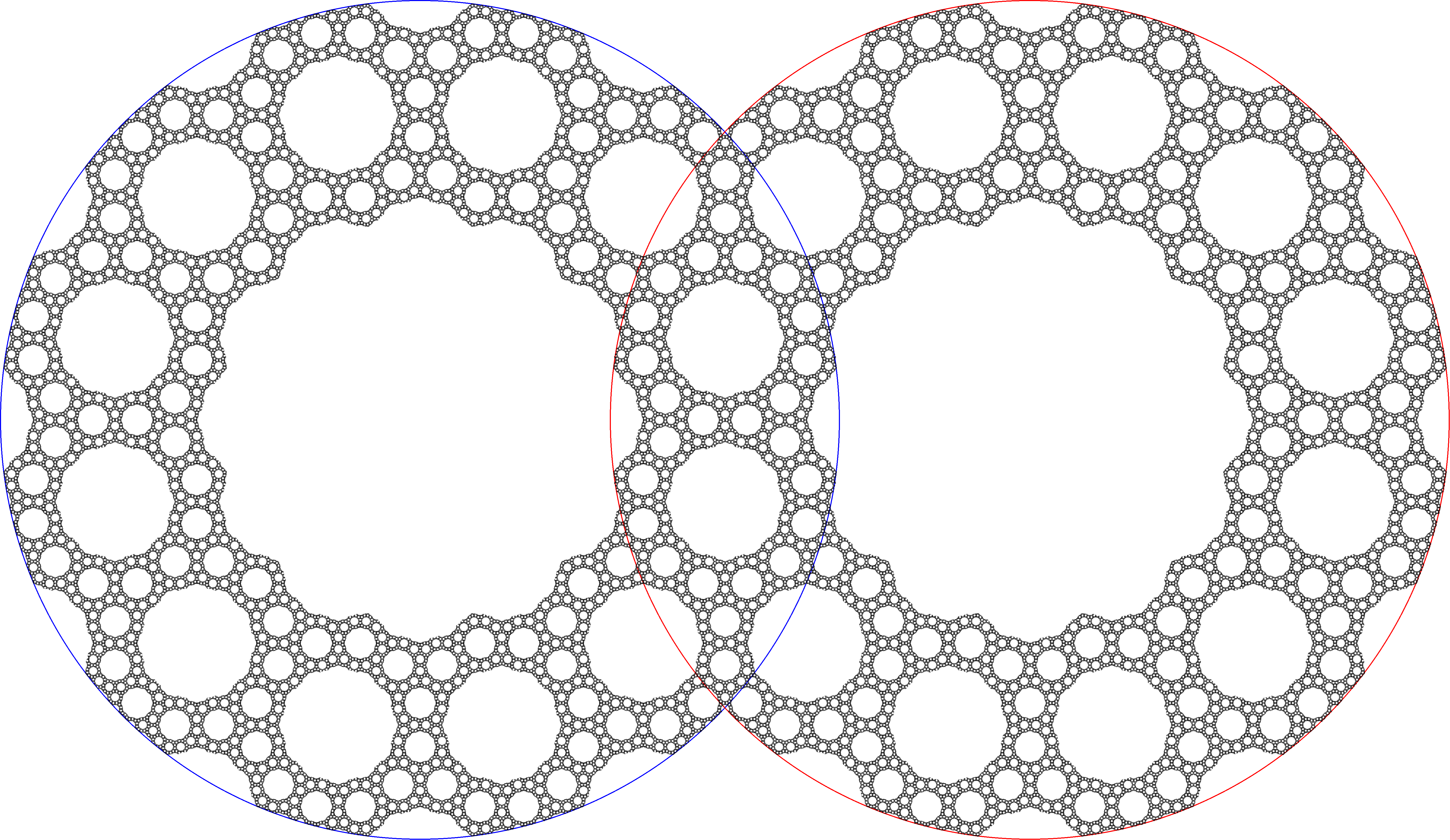}
    \caption{$GG_{12}(1.376547)$.}
\end{figure}

\begin{figure}[h!tbp]
	\centering
	\includegraphics[width=\textwidth]{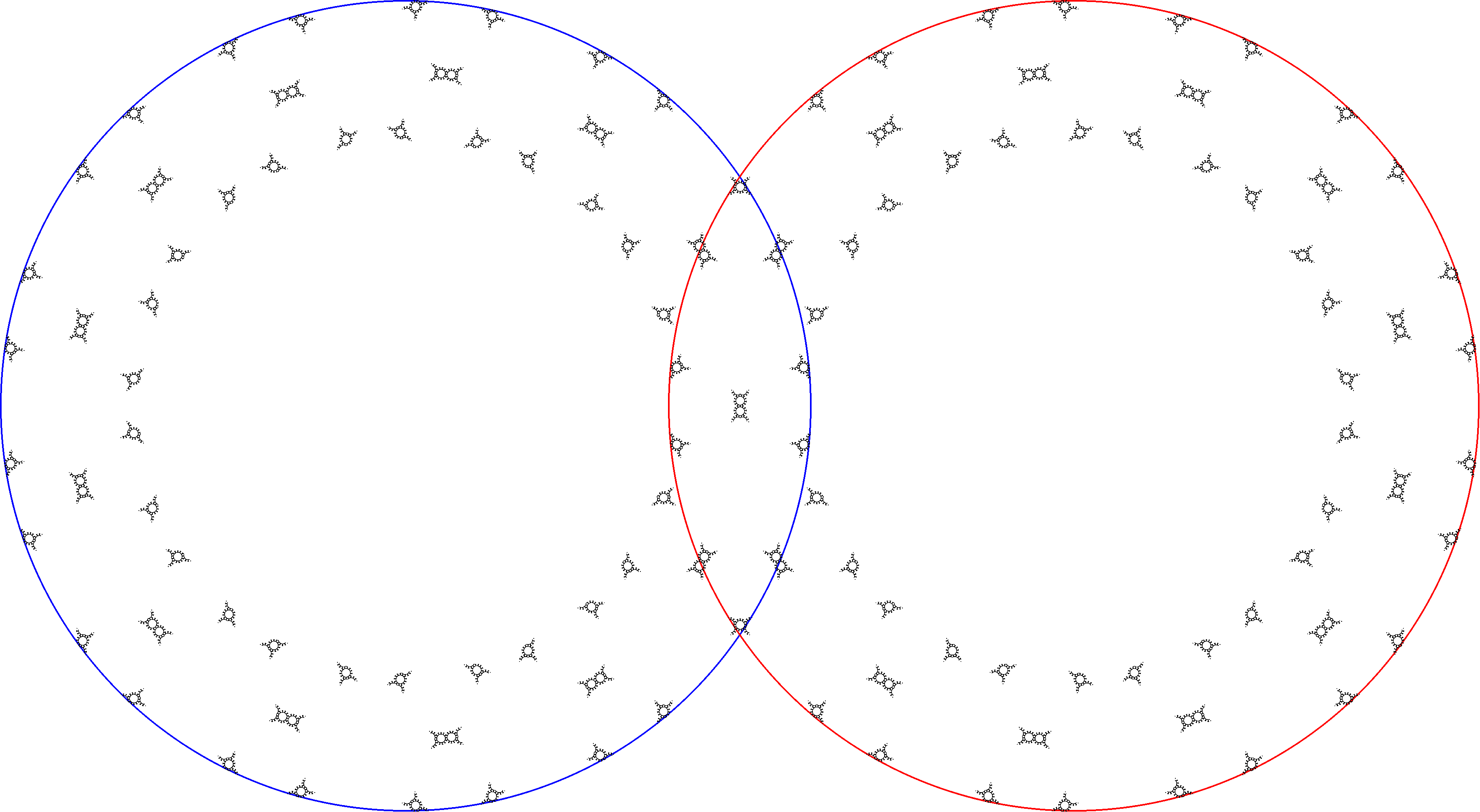}
    \caption{$GG_{13}(1.213594)$.}
\end{figure}

\begin{figure}[h!tbp]
	\centering
	\includegraphics[width=\textwidth]{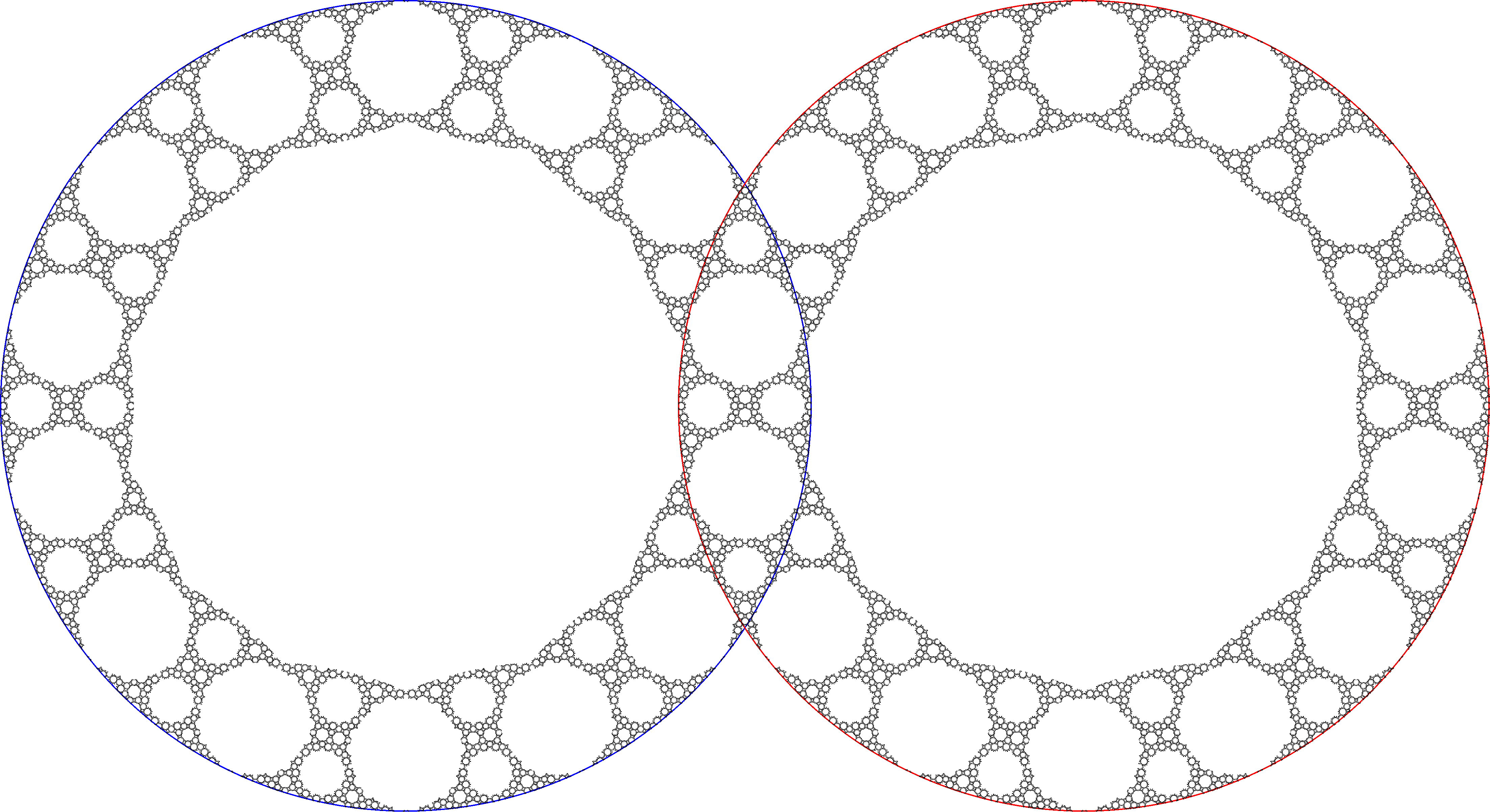}
    \caption{$GG_{14}(1.196554)$.}
\end{figure}

\begin{figure}[h!tbp]
	\centering
	\includegraphics[width=\textwidth]{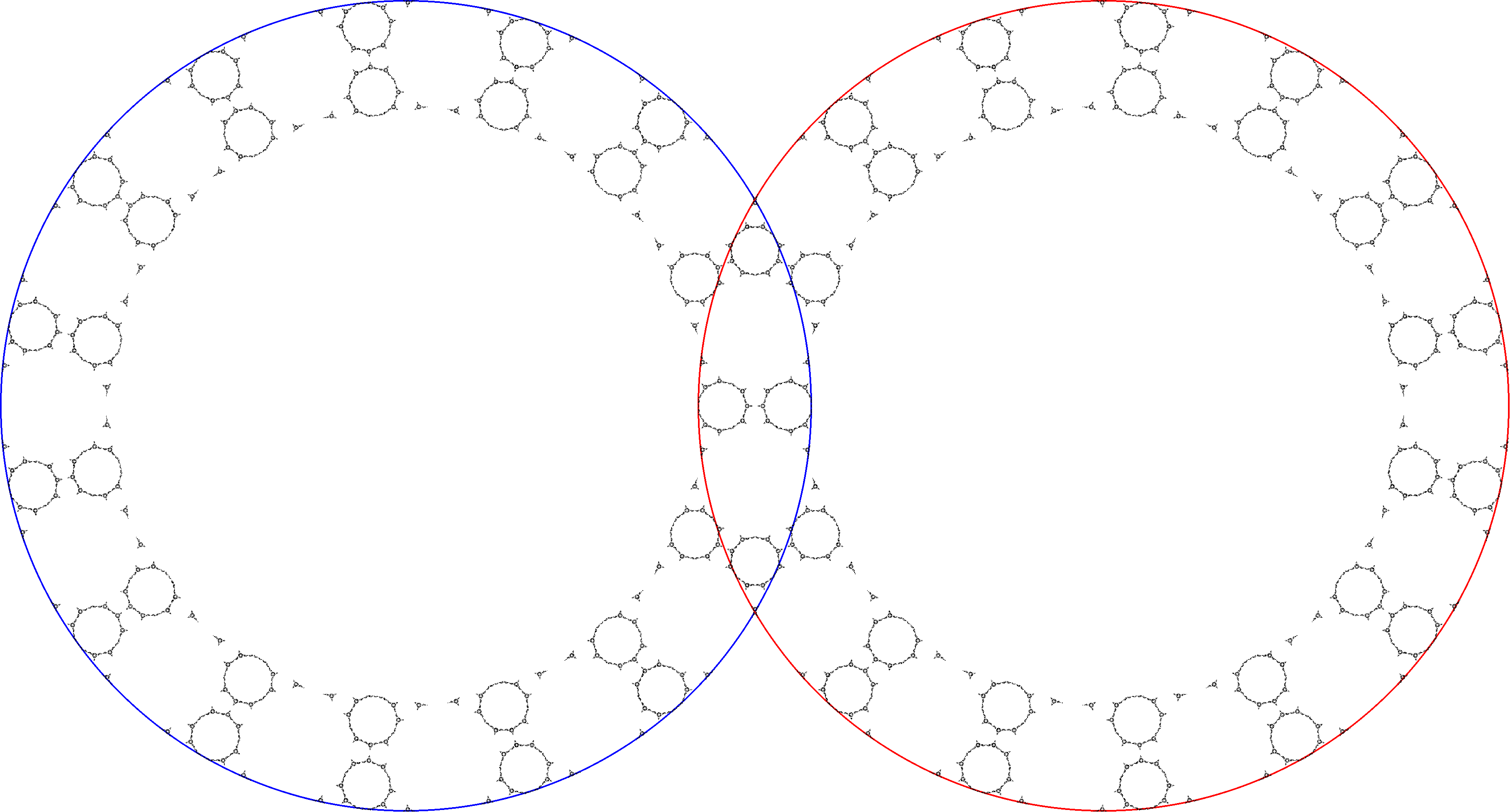}
    \caption{$GG_{15}(1.163275)$.}
\end{figure}

\begin{figure}[h!tbp]
	\centering
	\includegraphics[width=\textwidth]{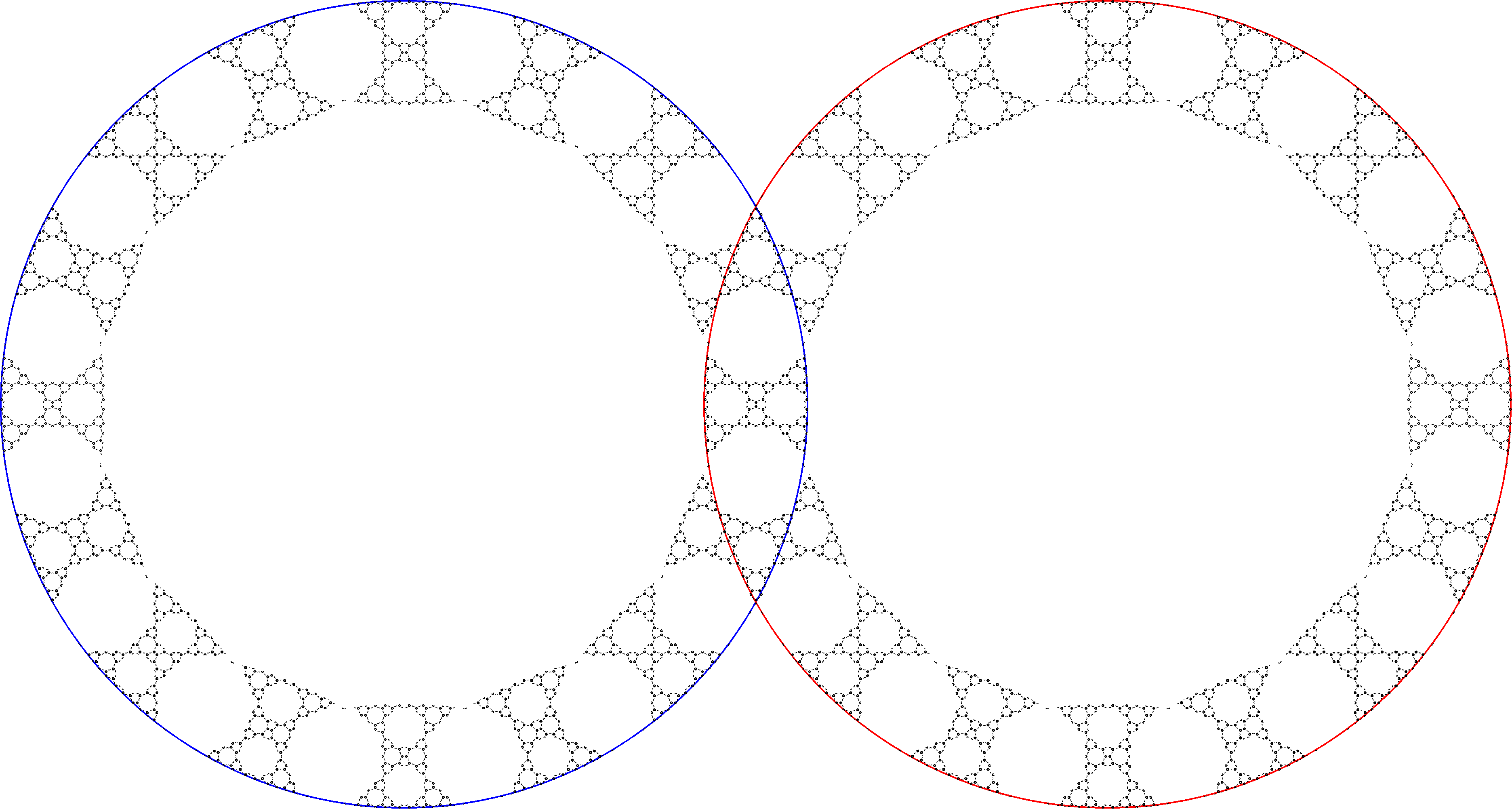}
    \caption{$GG_{16}(1.1484703)$.}
\end{figure}

\begin{figure}[h!tbp]
	\centering
	\includegraphics[width=\textwidth]{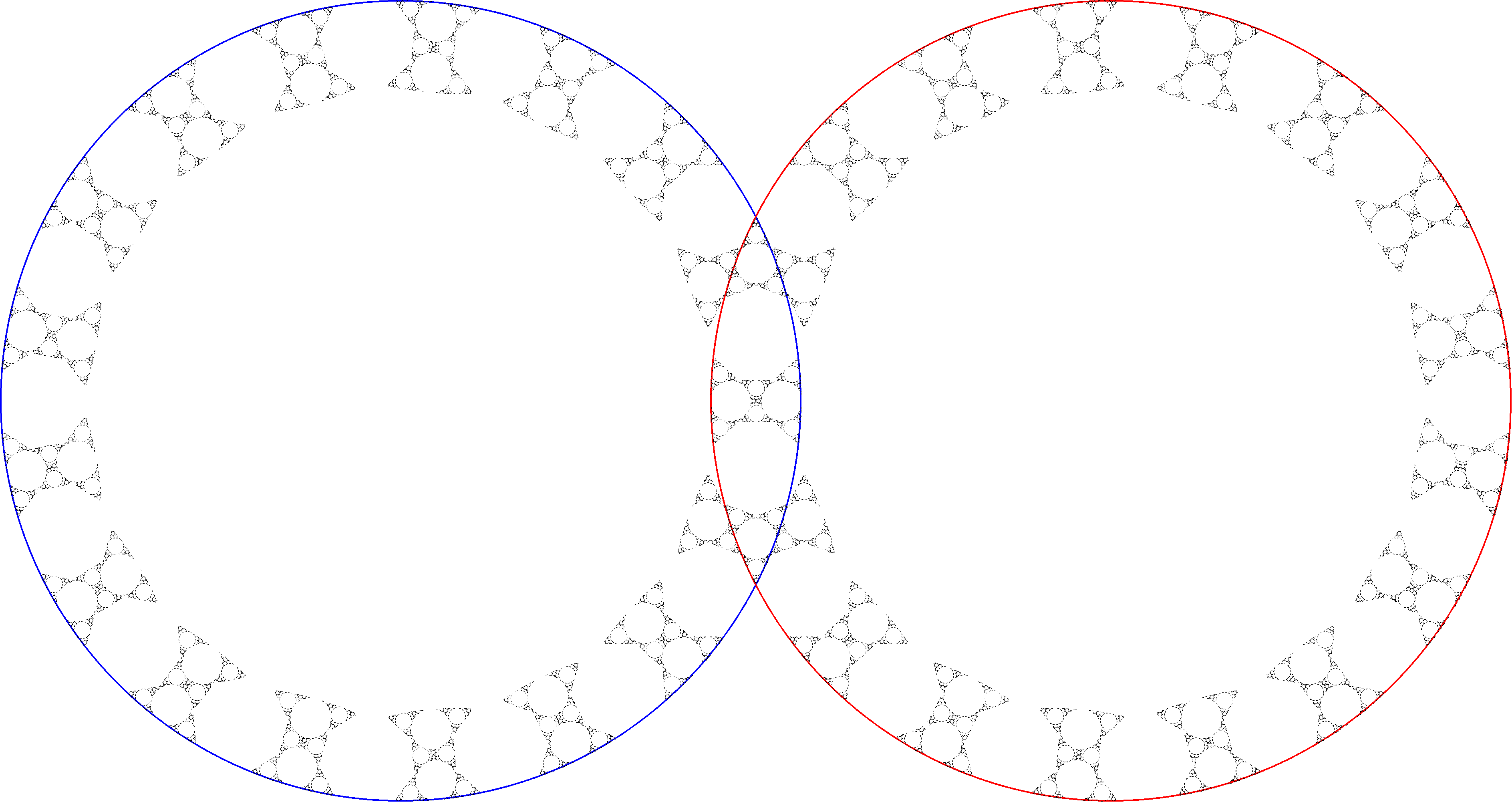}
    \caption{$GG_{17}(1.127509)$.}
\end{figure}

\begin{figure}[h!tbp]
	\centering
	\includegraphics[width=\textwidth]{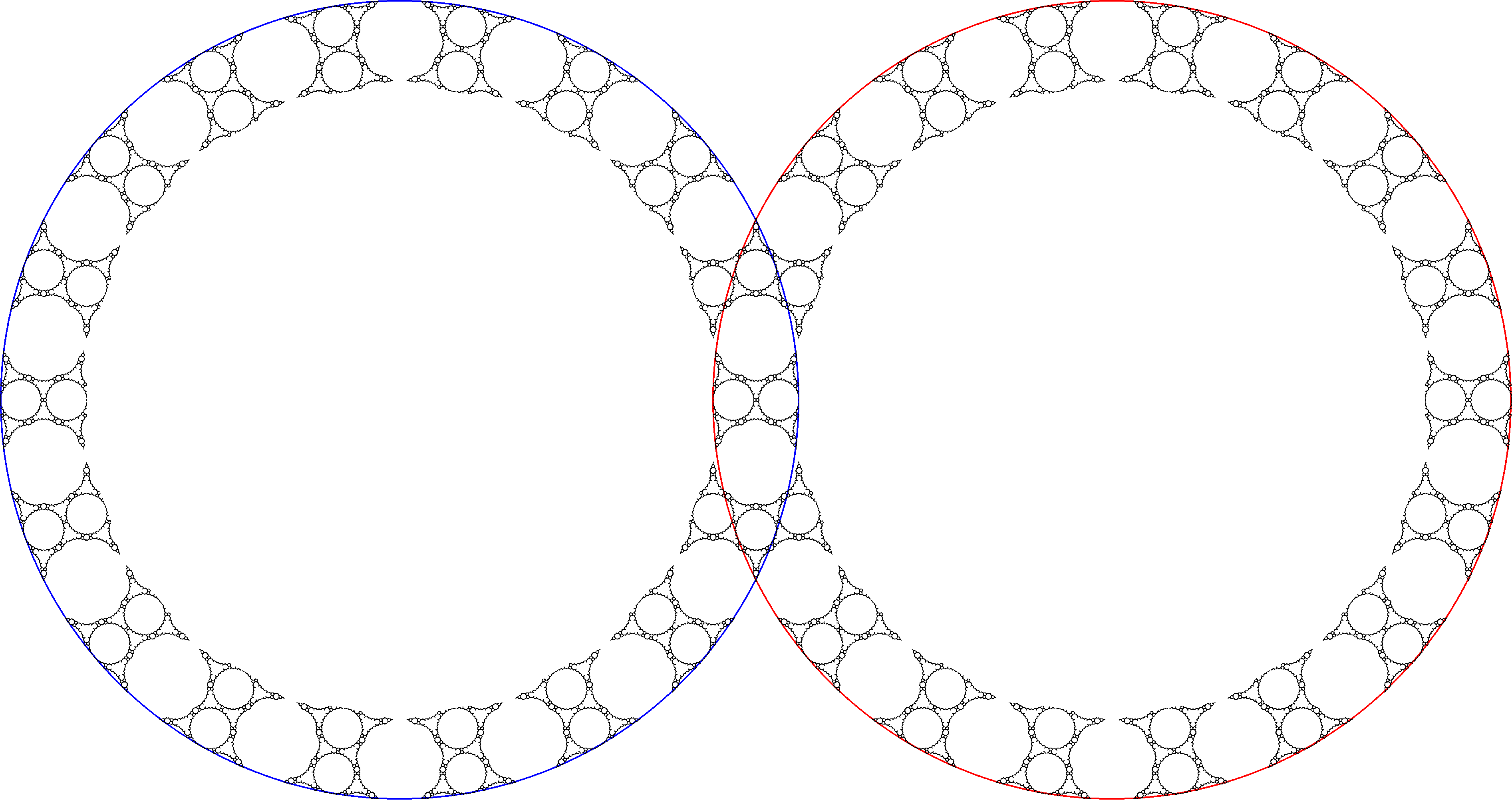}
    \caption{$GG_{18}(1.121504)$.}
\end{figure}

\begin{figure}[h!tbp]
	\centering
	\includegraphics[width=\textwidth]{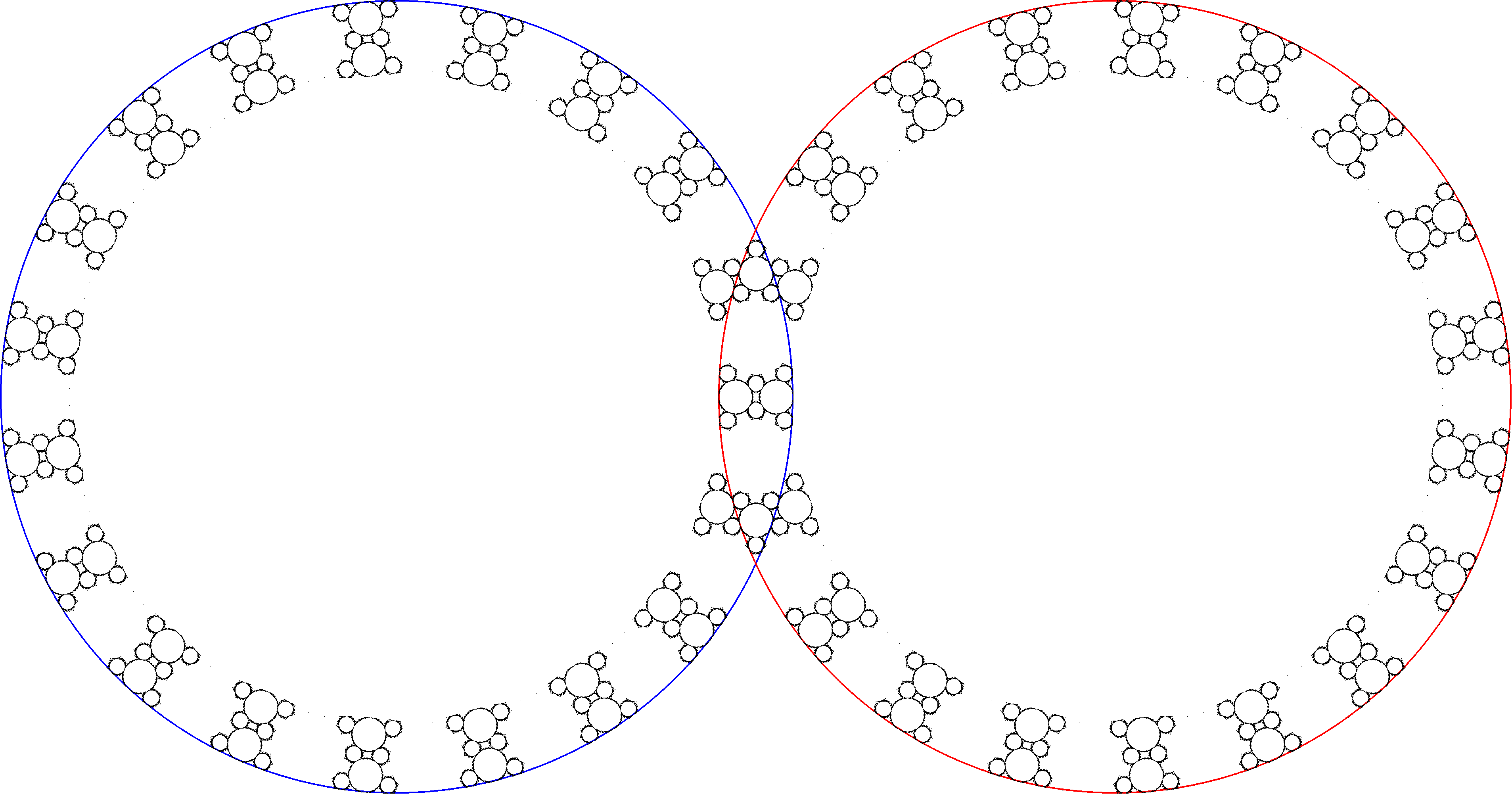}
    \caption{$GG_{19}(1.104246)$.}
\end{figure}

\begin{figure}[h!tbp]
	\centering
	\includegraphics[width=\textwidth]{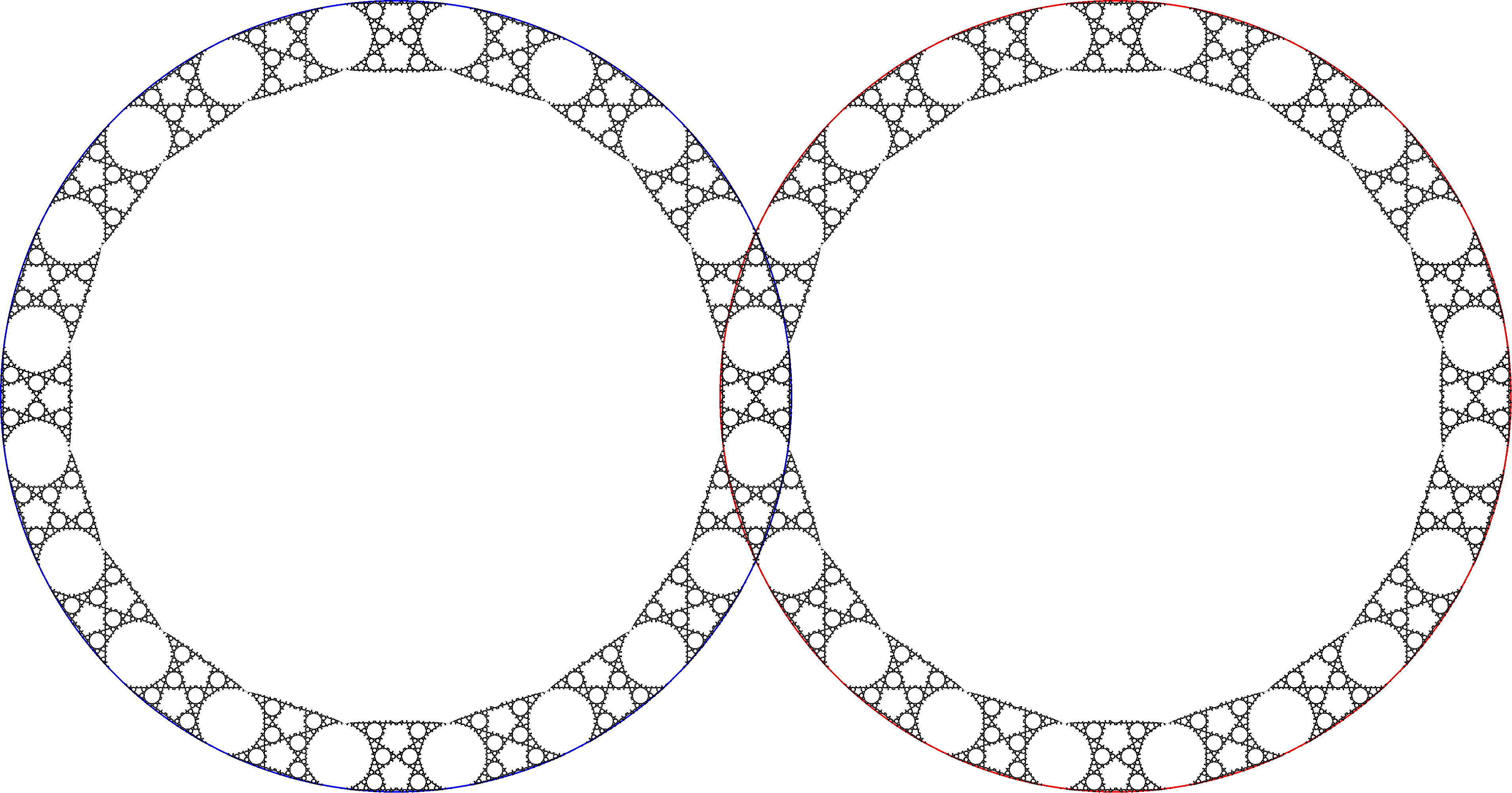}
    \caption{$GG_{20}(1.100581)$.}
\end{figure}

\end{document}